\documentclass{amsart}
\pdfoutput=1
\usepackage{amsmath,amssymb,stmaryrd,amsthm}
\usepackage{graphicx}
\usepackage{accents}

\newcommand{\pp}[2]{\frac{\partial #1}{\partial #2}} 

\newcommand{\dd}[2]{\frac{\diff#1}{\diff#2}}

\DeclareMathOperator{\diff}{d}

\DeclareMathOperator{\Diff}{Diff}
\DeclareMathOperator{\argmax}{argmax}
\newtheorem{theorem}[subsection]{Theorem}
\newtheorem{definition}[subsection]{Definition}
\newtheorem{lemma}[subsection]{Lemma}
\newtheorem{proposition}[subsection]{Proposition}
\newtheorem{corollary}[subsection]{Corollary}
\newtheorem{remark}[subsection]{Remark}

\usepackage{color}
\usepackage[margin=2cm]{geometry}

\newcommand{\revised}[1]{{#1}}

\usepackage{fancybox}
\usepackage{xspace}
\usepackage{hyperref}

\newcommand{\CC}{\ensuremath{\operatorname C}\xspace}
\newcommand{\HH}{\ensuremath{\operatorname H}\xspace}
\newcommand{\LL}{\ensuremath{\operatorname L}\xspace}

\newcommand{\WW}{\ensuremath{\operatorname W}\xspace}

\newcommand{\cont}[1]{\ensuremath{\CC^{#1}}}

\newcommand{\leb}[1]{\ensuremath{\LL^{#1}}}

\newcommand{\sob}[2]{\ensuremath{\WW^{#1,#2}}}
\newcommand{\sobz}[2]{\ensuremath{{\WW}_{0}^{#1,#2}}}
\newcommand{\sobh}[1]{\ensuremath{\HH^{#1}}}

\newcommand{\pddt}{\ensuremath{\frac{\partial}{\partial t}}}
\newcommand{\ddt}{\ensuremath{\frac{\diff}{\diff t}}}
\newcommand{\ddx}{\ensuremath{\frac{\partial}{\partial x}}}
\newcommand{\norm}[1]{\ensuremath{\left|#1\right|}}

\newcommand{\qp}[1]{\ensuremath{\!\left({#1}\right)}}

\newcommand{\reals}{\ensuremath{\mathbb R}\xspace}
\newcommand{\naturals}{\ensuremath{\mathbb N}\xspace}

\renewcommand*{\dot}[1]{\accentset{\mbox{\large\bfseries .}}{#1}}

\author{
  Colin J. Cotter
}
\address{
  Colin J. Cotter
  \thanks{
    Department of Mathematics, Imperial College London, UK.
    {\tt{colin.cotter@imperial.ac.uk}}.
}}

\author{
  Jacob Deasy
}
\address{
  Jacob Deasy
    \thanks{
      Department of Computer Science and Technology, University of Cambridge, UK.     {\tt{jd645@cam.ac.uk}}.
}}

\author{
  Tristan Pryer
}
\address{
  Tristan Pryer
  \thanks{
    Department of Mathematical Sciences, University of Bath, Bath BA2 7AY, UK.
    {\tt{tmp38@bath.ac.uk}}.
}}

\thanks{ TP was partially supported through the EPSRC grant EP/P000835/1.
  CJC was partially supported through the EPSRC grant EP/R029423/1. This
  support is gratefully acknowledged. }

\begin{document}
\title{The $r$-Hunter-Saxton equation, Smooth and Singular Solutions
  and their Approximation}
\maketitle

\begin{abstract}
  In this work we introduce the $r$-Hunter-Saxton equation, a
  generalisation of the Hunter-Saxton equation arising as extremals of
  an action principle posed in $\leb{r}$. We characterise solutions to
  the Cauchy problem, quantifying the blow-up time and studying
  various symmetry reductions. We construct piecewise linear functions
  and show that they are weak solutions to the $r$-Hunter-Saxton
  equation.
\end{abstract}

\section{Introduction}

The Hunter-Saxton (HS) equation is a 1+1 dimensional, variational,
partial differential equation (PDE), originally introduced as a model for
the propagation of waves in the director field of a nematic liquid
crystal \cite{hunter1991dynamics, hunter1995nonlinear,
  hunter1995bnonlinear}. This problem arises in three different forms,
related to each other \revised{formally} through differentiation,
\begin{align}
  \label{eq:HS1}
  u_t + \left(\frac{u^2}{2}\right)_x & = \frac{1}{4}
  \left(
  \int_{-\infty}^x - \int_x^\infty\right)u^2_x \diff x, \\
  \label{eq:HS2}
  (u_t + uu_x)_x & = \frac{1}{2}u^2_x, \\
  \label{eq:HS3}
  u_{txx} + 2u_xu_{xx} + uu_{xxx} & = 0,
\end{align}
where the subindices denote partial differentiation with respect to
the corresponding independent variable. These equations can be
considered over a domain $\Omega$ that is either the real line, the
periodic interval or the interval $[a,b]$ with zero boundary
conditions $u(a)=u(b)=0$.

The HS equation also has an important geometric interpretation as it
describes the geodesic flow on the diffeomorphism group on the domain
$\Omega$, with the right-invariant metric defined through the $\sobh
1$-inner product
\begin{equation}
  \langle f, g \rangle = \int_\Omega f_x g_x \diff x.
\end{equation}
Consider time-parameterised diffeomorphisms $g\in\Diff(\Omega)$ with
\begin{equation}
  \label{eq:gu}
  g_t(x,t) = u(g(x,t),t),
\end{equation}
then the HS equations can formally be interpreted as
extreme points for the action principle
\begin{equation}
  \delta \int_0^T l[u]  \diff t = 0,
\end{equation}
with
\begin{equation}
  l[u] = \int_\Omega \frac{1}{2}u_x^2 \diff x,
\end{equation}
and endpoint conditions $\delta u(x,0)=\delta u(x,T)=0$. This action
seeks a geodesic path of diffeomorphisms $g(x,t)$ with $g(x,0)=g_1(x)$ and
$g(x,T)=g_2(x)$ such that the distance functional $l[u]$ is
extremised. This problem is the one-dimensional version of problems
that arise in computational anatomy \cite{miller2006geodesic}.
Then, \eqref{eq:HS3} is obtained as the corresponding
Euler-Poincar\'e equation for this action principle after observing
that \eqref{eq:gu} implies that we must have constrained variations
$\delta u = w_t + wu_x - uw_x$ for arbitrary $w$. The conserved energy
then arises from Noether's theorem applied to time translation
symmetry; we also obtain a Poisson structure for \eqref{eq:HS3}
through this route. The Euler-Poincar\'e derivation is a formal
calculation that was made rigorous in \cite{khesin2003euler}.
\revised{A useful connection between these geodesics on the periodic
  interval and the
$\leb{2}$-sphere was made in \cite{lenells2007hunter, lenells2008hunter}
  \emph{via} an explicit mapping; \cite{bauer2014homogeneous} extended
  this technique to Hunter-Saxton geodesics on the real line.}

The integrability of the HS equation was shown in
\cite{hunter1994completely} by relating it the Camassa-Holm equation
\begin{equation}
  \label{eq:CH}
  u_t - u_{xxt} + 3uu_x - 2u_x u_{xx} - uu_{xxx} = 0.
\end{equation}
Indeed, it can be shown that (\ref{eq:HS3}) arises from a high
frequency limit of (\ref{eq:CH}) and, thus, has a bi-Hamiltonian
structure \cite{ali2009orientation}. Further connections can be made
through the geometric interpretation as the Camassa-Holm equation
describes geodesic flow with respect to the right-invariant metric
\cite{kouranbaeva1999camassa}
\begin{equation}
  \langle f, g \rangle = \int_\Omega f_x g_x + f g \diff x.
\end{equation}
The HS equation has a conserved energy,
\begin{equation}
  \ddt E_1 := \ddt \int_\Omega u_x^2 \diff x = 0,
\end{equation}
which represents one of the two Hamiltonians of the problem, the other
being
\begin{equation}
  \ddt E_2 := \ddt \int_\Omega u u_x^2 \diff x = 0.
\end{equation}
The explicit control of these two functionals ensures some regularity
of the solution. 

Solutions to this equation can even be written down explicitly. The HS
equation given in the forms (\ref{eq:HS1}-\ref{eq:HS2}) has piecewise
linear weak solutions that conserve this energy
\cite{hunter1994completely}; these solutions can be described by a set
of moving points plus the value of $u$ at those points with linear
interpolation in between. In \cite{bressan2005global} it was shown how
to make sense of these piecewise solutions as weak solutions, and that
weak solutions of HS are locally Lipschitz with respect to the initial
conditions, making use of techniques from optimal transport. The
locality is clear since it is possible to find piecewise linear
solutions where two points with different $u$ values collide in finite
time, leading to a jump in the solution.

In this paper, we investigate the effect of modifying the distance
functional $l[u]$ such that
\begin{equation}
  l[u] = \int_\Omega \frac{1}{r} \norm{u_x}^r \diff x,
\label{eq:l_r}
\end{equation}
corresponding to the $\sob{1}{r}(\Omega)$ distance, for general $r$ rather
than for $r=2$.
\revised{We call the
  resulting equation the $r$-Hunter-Saxton ($r$-HS) equation, which
  takes the form
\begin{equation}
  \qp{\norm{u_x}^{r-2}u_x}_{xt}
  +
  \qp{\norm{u_x}^{r-2}u_x}_xu_x
  +
  \qp{\qp{\norm{u_x}^{r-2}u_x}_xu}_x = 0,
\end{equation}}
To the author's knowledge, previous work on
quantifying the nature of solutions to PDEs arising from energy
functionals posed on $\leb{r}$ is limited to the elliptic
case. Indeed, a sequence of works initiated by \cite{Aronsson:1965}
focussed on the study of the Euler-Lagrange equations of
(\ref{eq:l_r}) in the multi-dimensional setting as $r\to \infty$, see
\cite{Katzourakis:2015,Pryer:2018} for an accessible overview. In
\cite{KatzourakisPryer:2016} the authors obtained results for the
vectorial analogy of (\ref{eq:l_r}) and in
\cite{KatzourakisPryer:2016,KatzourakisMoser:2017,KatzourakisPryer:2018b}
a minimisation problem involving second derivatives was examined. In
this work, for the first time, we examine the effect modifying the
distance functional has on an evolution problem. One of the motivating
reasons to modify the distance functional is that a certain amount of
regularity is required for a solution to exist to $\eqref{eq:gu}$
\cite{dupuis1998variational}, specifically $u(x,t)$ is required to be
in $\leb{1}([0,T], \sob{1}{\infty}(\Omega))$. For the HS equations
with $r=2$, this is not satisfied in general and we see the loss of
the diffeomorphism property when the piecewise-linear solutions blow
up as above. With that in mind, we are interested in how the solutions
to the resulting PDE behave as $r\to\infty$. 

To extend the notion of piecewise linear solutions of the HS equation
to the $r$-HS equation we use an optimal control formulation that
arises when trying to optimise $l[u]$ such that a set of points moving
with $u$ are transported from one configuration to another. We shall
see that the optimal $u$ is then piecewise linear. Following
\cite{cotter2009continuous,gay2011clebsch}, after eliminating $u$, we
obtain a Hamiltonian system for the point locations $q$ and their
conjugate momenta $p$, with the conserved energy being equivalent to
$l[u]$.  For $u$ in a finite dimensional state space, it can be shown
that after eliminating $p$ and $q$, $u$ satisfies the Euler-Poincar\'e
equation corresponding to $u$. However, in this particular case, the
$r$-HS equation does not make sense because piecewise solutions do not
have enough regularity. In this paper, we resolve this situation by
showing that the piecewise linear solutions are weak solutions of
another equation whose formal derivative gives the r-HS equation.

The rest of this paper is organised as follows. In Section
\ref{sec:rHS} we derive the $r$-HS equation. In Section \ref{sec:sym}
we study some fundamental properties of the $r$-HS equation,
specifically we characterise smooth solutions to the problem, give a
blow-up criteria and give examples of special solutions arising from a
symmetry reduction technique. In Section \ref{sec:linear} we introduce
piecewise-linear functions as solutions to an optimal control problem
for points on an interval. In Section \ref{sec:weak} we show that
these piecewise-linear functions are weak solutions of an integrated
form of the $r$-HS equation.  In section \ref{sec:numerics} we
calculate some numerical examples, and finally in \ref{sec:outlook} we
provide a summary and outlook.

\section{The $r$-Hunter-Saxton equation}
\label{sec:rHS}

In this section we introduce the $r$-HS equation, describe some of its
properties.

\begin{definition}[$r$-Hunter-Saxton equation]
  The $r$-Hunter-Saxton equation on the interval $[a,b]$ in its most
  general form is given by
  \begin{equation}
    \label{eq:m_t}
    \qp{\norm{u_x}^{r-2}u_x}_{xt}
    +
    \qp{\norm{u_x}^{r-2}u_x}_xu_x
    +
    \qp{\qp{\norm{u_x}^{r-2}u_x}_xu}_x = 0,    
  \end{equation}
  with boundary conditions $u(a)=u(b)=0$.
\end{definition}

One should notice that, as with the ($2$)-HS equation, the general
$r$-HS equation has various equivalent forms.
\begin{proposition}[Different forms of (\ref{eq:m_t})]
  The following are formally equivalent formulations of (\ref{eq:m_t}):
  \begin{equation}
    \label{eq:rHS1}
    \qp{\norm{u_x}^{r-2}u_x}_{xt}
    +
    \qp{\norm{u_x}^{r-2}u_x}_xu_x
    +
    \qp{\qp{\norm{u_x}^{r-2}u_x}_xu}_x
    =
    0,
  \end{equation}
  \begin{equation}
    \label{eq:rHS2}
    \norm{u_x}^{r-2}u_{xt}
    +
    \frac{1}{r} {\norm{u_x}^{r}}
    +
    {\norm{u_x}^{r-2} u_{xx} u}
    = c(t),
  \end{equation}
  and
  \begin{equation}
    \label{eq:rHS3}
    \qp{\norm{u_x}^{r-2}u_x}_{t}
    +
    \qp{\norm{u_x}^{r-2} u u_x}_x
    =
    \frac 1 r {\norm{u_x}^{r}} + c(t)
    ,
  \end{equation}
  for some $c(t)$ that ensures $u(b)=0$. Notice that setting $r=2$, in
  (\ref{eq:rHS3}) yields exactly (\ref{eq:HS2}).
\end{proposition}
\begin{proof}
  Firstly note that
  \begin{equation}
    \ddx \qp{\norm{u_x}^r}
    =
    r \norm{u_x}^{r-2} u_x u_{xx},
  \end{equation}
  and hence
  \begin{equation}
    \begin{split}
      \ddx \qp{\norm{u_x}^{r-2}u_x}
      &=
      \qp{r-2} \norm{u_x}^{r-4} u_x^2 u_{xx} + \norm{u_x}^{r-2} u_{xx},
      \\
      &=
      \qp{r-1} \norm{u_x}^{r-2} u_{xx}.
    \end{split}
  \end{equation}
  Similarly,
  \begin{equation}
    \begin{split}
      \pddt \qp{\norm{u_x}^{r-2}u_x}
      &=
      \qp{r-1} \norm{u_x}^{r-2} u_{xt}.
    \end{split}
  \end{equation}
  Making use of these we see that
  \begin{equation}
    \begin{split}
      0
      &=
      \qp{\norm{u_x}^{r-2}u_x}_{tx}
      +
      \qp{\norm{u_x}^{r-2}u_x}_xu_x
      +
      \qp{\qp{\norm{u_x}^{r-2}u_x}_xu}_x,
      \\
      &=
      \qp{r-1}\qp{\norm{u_x}^{r-2}u_{xt}}_{x}
      +
      \qp{r-1} \norm{u_x}^{r-2} u_{xx} u_x
      +
      \qp{r-1} \qp{\norm{u_x}^{r-2} u_{xx} u}_x,
      \\
      &=
      \qp{r-1}\qp{\norm{u_x}^{r-2}u_{xt}}_{x}
      +
      \frac{\qp{r-1}}{r} \qp{\norm{u_x}^{r}}_x
      +
      \qp{r-1} \qp{\norm{u_x}^{r-2} u_{xx} u}_x.
    \end{split}
  \end{equation}
  Hence (\ref{eq:rHS1}) is the formal derivative of (\ref{eq:rHS2}).

  To see (\ref{eq:rHS3}) note that
  \begin{equation}
    \begin{split}
      \ddx \qp{\norm{u_x}^{r-2}u u_x}
      &=
      \qp{r-2} \norm{u_x}^{r-4} u_x^2 u_{xx} u
      +
      \norm{u_x}^{r-2} u_{x}^2
      +
      \norm{u_x}^{r-2} u u_{xx},
      \\
      &=
      \qp{r-1} \norm{u_x}^{r-2} u u_{xx} + \norm{u_x}^r.
    \end{split}
  \end{equation}
  Hence
  \begin{equation}
    \begin{split}
      c(t)
      &=
      \qp{\norm{u_x}^{r-2}u_x}_{t}
      +
      \frac{\qp{r-1}}{r} {\norm{u_x}^{r}}
      +
      \qp{r-1} {\norm{u_x}^{r-2} u_{xx} u},
      \\
      &=
      \qp{\norm{u_x}^{r-2}u_x}_{t}
      -
      \frac{1}{r} {\norm{u_x}^{r}}
      +
      \qp{\norm{u_x}^{r-2} u u_x}_x,
    \end{split}
  \end{equation}
  as required.
\end{proof}


The equations can also be defined on the real line, or with periodic
boundary conditions, but we concentrate on the boundary value problem
in this paper for simplicity.

\revised{Next we show that the r-Hunter-Saxton equation emerges from
  Hamilton's principle for a $\sob{1}{r}$ Lagrangian.}
\begin{proposition}[Euler-Poincar\'e equation]
  The $r$-Hunter-Saxton \eqref{eq:m_t} is the Euler-Poincar\'e
  equation for the Lagrangian
  \begin{equation}
    l[u] = \int_\Omega \frac{1}{r} \norm{u_x}^r \diff x.
  \end{equation}
\end{proposition}
\begin{proof}
  We have
  \begin{align}
    \label{eq:deltaS}
    \delta S
    & =
    \delta \int_0^T \int_a^b \frac{1}{r} \norm{u_x}^r \diff x \diff t, \\
    & =
    \int_0^T \int_a^b \norm{u_x}^{r-2} u_x\delta u_x \diff x \diff t, \\
    & =
    \int_0^T \int_a^b -\qp{\norm{u_x}^{r-2}u_x}_x\delta u \diff x \diff t,
  \end{align}
  through an integration by parts making use of the boundary
  conditions $u(a,t)=u(b,t)=0$. Following \cite{HoMaRa98},
  \eqref{eq:gu} implies that
  \begin{equation}
    \label{eq:deltau}
    \delta u = {w}_t - wu_x + uw_x,
  \end{equation}
  with $w(x,0)=w(x,T)=0$.  Substituting (\ref{eq:deltau}) into
  (\ref{eq:deltaS}), we obtain
  \begin{align}
    \label{eq:EP weak}
    0 =
    \delta S 
    & =
    \int_0^T \int_a^b -\qp{\norm{u_x}^{r-2}u_x}_x\qp{{w}_t-wu_x + uw_x}\diff x \diff t,
    \\
    & = \int_0^T \int_a^b
    \qp{
      \qp{\norm{u_x}^{r-2}u_x}_{xt}
      +
      \qp{\norm{u_x}^{r-2}u_x}_xu_x
      +
      \qp{\qp{\norm{u_x}^{r-2}u_x}_xu}_x
    }
    w\diff x \diff t,
  \end{align}
  which is satisfied by solutions of \eqref{eq:m_t} for arbitrary $w$, as required.
\end{proof}


\revised{Later we shall make use of the following weak form of \eqref{eq:rHS3}.}
\begin{definition}[Weak integrated r-HS equation]
  Let $\sobz{1}{r}(a,b)$ be the space of functions with
  \begin{equation}
    \int_a^b \norm{u}^r + \norm{u_x}^r \diff x < \infty,
  \end{equation}
  that satisfy the boundary conditions $u(a) = u(b) = 0$.  Then, $u\in
  \sobz{1}{r}(a,b)$ is a weak solution of (\ref{eq:rHS3}) if it
  satisfies
  \begin{equation}
    \label{eq:weak rHS}
    \int_a^b
    \qp{
      \qp{\norm{u_x}^{r-2}u_x}_t
      - \frac{1}{r}\norm{u_x}^r}\phi
    -
    u\norm{u_x}^{r-2}u_x\phi_x \diff x =  c(t) \int_a^b \phi \diff x,
  \end{equation}
  for all test functions $\phi\in \sobz{1}{r}(a,b)$.
\end{definition}

\section{Symmetries and classical solutions}
\label{sec:sym}

In this section we examine some symmetries and characterise classical
solutions and their existence time for the $r$-HS equation. We take
inspiration from the arguments in \cite[\S 3]{hunter1991dynamics} and
show that, remarkably, many of the properties shown for the $r=2$ case
generalise for arbitrary $r$. We begin by examining a characteristic
reduction. To that end, throughout this section we will assume $u \in
\cont{2}(\mathbb{R})$ be a classical solution of the Cauchy problem
  \begin{equation}
    \label{eq:rHS-Cauchy}
    \begin{split}
      \qp{\norm{u_x}^{r-2}u_x}_{t}
      +
      \qp{\norm{u_x}^{r-2} u u_x}_x
      &=
      \frac 1 r {\norm{u_x}^{r}},
      \\
      u(x,0) &= u_0(x).
    \end{split}
  \end{equation}

  \begin{theorem}[Smooth solution characterisation]
    \revised{
      Every smooth solution of (\ref{eq:rHS-Cauchy}) can be written implicitly as 
      \begin{equation}
        \label{eq:imp-soln}
        \begin{split}
          u &= H'(t) + \sum_{k=1}^r k C_k t^{k-1} G_k(\xi),
          \\
          x &= H(t) + \xi + \sum_{k=1}^r C_k t^k G_k(\xi).
        \end{split}
      \end{equation}
      where $H \in \cont{1}(\reals)$ is any function with $H(0) = H'(0)
      = 0$ and, for $k \in \naturals$, $G_k \in \cont{2}(\reals)$ is
      defined by requiring
      \begin{equation}
        G_k'(\xi) = u_0'(\xi)^k.
      \end{equation}
      The constants $C_k$ are defined as
      \begin{equation}
        C_k := \frac{r!}{r^k k!(r-k)!}.
      \end{equation}
      }
  \end{theorem}
  \begin{proof}
    \revised{
    Let $\xi$ denote a characteristic curve with $U(\xi,t) =
    u(X(\xi,t),t)$ and $X$ satisfying the initial value problem
    \begin{equation}
      \begin{split}
        X_t(\xi, t) = U(\xi, t),
        \\
        X(\xi, 0) = \xi.
      \end{split}
    \end{equation}
    Further, let $V(\xi, t) = X_{\xi}(\xi, t)$. It can then be verified
    that $V$ satisfies
    \begin{equation}
      \begin{split}
        V_t
        &=
        u_X X_\xi,
        \\
        \qp{\norm{V_t}^{r-2}V_t}_t
        &=
        \qp{r-1} \qp{\norm{X_{\xi} u_X}^{r-2} X_{\xi}
          \qp{u_{XX} u + u_{Xt} + u^2_X }},
      \end{split}
    \end{equation}
    and hence solves the second order initial value problem
    \begin{equation}
      \begin{split}
        V \qp{\norm{V_t}^{r-2}V_t}_t &= \frac{\qp{r-1}^2}{r} \norm{V_t}^r,
        \\
        V(\xi,0) &= 1,
        \\
        V_t(\xi,0) &= u_0'(\xi).
      \end{split}
    \end{equation}
    This can then be solved to show that 
    \begin{equation}
      V = \qp{\frac{tu_0'(\xi)}{r} + 1}^r.
    \end{equation}
    This should be compared with the $r=2$ case found in \cite[\S
      3]{hunter1991dynamics}. The final result follows from solving
    the system
    \begin{equation}
      \begin{split}
        U(\xi, t) &= X_t(\xi, t),
        \\
        X_{\xi}(\xi, t) &= \qp{\frac{tu_0'(\xi)}{r} + 1}^r,
      \end{split}
    \end{equation}
    for $U$ to show (\ref{eq:imp-soln}).

    It is also clear that any functions satisfying (\ref{eq:imp-soln})
    are indeed solutions of (\ref{eq:rHS-Cauchy}) so long as the
    relation for the characteristic curve is invertible. In view of
    this, it is expected that smooth solutions break in finite time. }
  \end{proof}
  \begin{corollary}[Existence and uniqueness]
    \revised{Suppose the Cauchy problem (\ref{eq:rHS-Cauchy}) is coupled with
    an asymptotic boundary condition that
    \begin{equation}
      \label{eq:ass-bc}
      \lim_{x\to \infty} u(x,t) = 0
    \end{equation}
    and let the initial condition $u_0 \in \sob{1}{r}(\reals)$ decay asymptotically, that is
    \begin{equation}
      \lim_{x\to \infty} u_0(x,t) = 0.
    \end{equation}
    Then, there is a $T>0$ such that for $t\in(0,T)$ there is a unique
    smooth solution of the Cauchy problem coupled to the asymptotic
    boundary condition (\ref{eq:ass-bc}) given by (\ref{eq:imp-soln}),
    $H(t) \equiv 0$ and
    \begin{equation}
      G_k(\xi) = -\int_\xi^\infty u_0'(s)^k \diff s \text{ for } k\in [1,r].
    \end{equation}}
  \end{corollary}
  \begin{theorem}[Blow up time]
    Let $T = \frac{r}{\sup_{x\in\Omega} -u_0'(\xi)}$ then, if $u_0$ is
    not monotonically increasing, (\ref{eq:rHS-Cauchy}) has a smooth
    solution for all $t\in (0,T)$ and $\sup \norm{u_x} \to \infty$ as
    $t\to T$.
  \end{theorem}
  \begin{proof}
    The implicit function theorem guarantees that as long as $X_\xi
    \neq 0$ there is a smooth solution of (\ref{eq:rHS-Cauchy}). Given
    \begin{equation}
      X_\xi = \qp{\frac{tu_0'(\xi)}{r} + 1}^r ,
    \end{equation}
    we see that $X_\xi = 0$ if and only if
    \begin{equation}
      t = \frac{r}{-u_0'(\xi)}.
    \end{equation}
    So $T = \frac{r}{\sup -u_0'(\xi)}$ as required.
  \end{proof}

  \begin{remark}[Relating to Burgers' equation]
    \revised{ At this point we note that the $1$-HS equation is
      formally equivalent to the inviscid Burgers' equation. Indeed, differentiating
      \begin{equation}
        u_t + u u_x = 0
      \end{equation}
      in space then dividing through by $\norm{u_x}$ yields
      (\ref{eq:rHS2}) with $r=1$. This means that for $1\leq r \leq
      2$, the $r$-HS equations can be viewed as an interpolation
      between Burgers' equation and the classic HS equation.}

    Notice the maximal smooth solution time of (\ref{eq:rHS-Cauchy})
    increases linearly as $r$ increases. In fact, it is exactly
    $r$-times the shock time for \revised{the} inviscid Burgers' equation. In
    particular, as $r\to\infty$, the blow up time $T\to\infty$. We will
    return to this point later in this work.
  \end{remark}
  
  \subsection*{Symmetries}
  \revised{In this section we find all
    possible exact solutions that result from inspecting the symmetries
    of the equation.
    For the rest of this section, for simplicity of
    computation, we make the assumption that $r\geq 2$ is even.}  A
  Lie point symmetry of equation (\ref{eq:rHS-Cauchy}) is a flow
  \begin{equation}
    \label{eq:flow}
    \qp{\widetilde{x},\widetilde{t},\widetilde{u}}
    =
    \qp{e^{\epsilon X}x,e^{\epsilon X}t,e^{\epsilon X}u},
  \end{equation}
  generated by a vector field
  \begin{equation}
    \label{eq:vf}
    X
    =
    \xi^1(x,t,u)\frac{\partial}{\partial x}
    +
    \xi^2(x,t,u)\frac{\partial}{\partial t}
    +
    \eta(x,t,u)\frac{\partial}{\partial u},
  \end{equation}
  such that $\widetilde{u}(\widetilde{x},\widetilde{y})$ is a solution
  of (\ref{eq:rHS-Cauchy}) whenever $u(x,y)$ is a solution of
  (\ref{eq:rHS-Cauchy}). As usual, we denote by $e^{\epsilon X}$ the
  \emph{Lie series} $\sum_{k=0}^{\infty}\frac{\epsilon^k}{k!}X^k$ with
  $X^{k}=XX^{k-1}$ and $X^0=1$.
  
  To find the symmetries of the $r$-HS equation we are required to
  solve the infinitesimal invariance condition for the vector field
  \eqref{eq:vf}. To do this we use the prolongation of $X$
  \cite{Bluman2008, Olver:1993, Stephani1989}. The infinitesimal
  symmetry condition decomposes to a large overdetermined system of
  linear PDEs for $\xi^1$, $\xi^2$ and $\eta$ known as
  \textit{determining equations}. \revised{The following three
    Propositions give the overdetermined system, the general form of
    the determining equations and the Lie algebra generators. These
    results were proven symbolically using the SYM package
    \cite{dimas2004sym,dimas2006new}. This procedure is described in
    further detail in the case of the $p$-Laplacian in
    \cite{PapamikosPryer:2019}}.

  \begin{proposition}[Infinitesimal invariance]
    The infinitesimal invariance condition is equivalent to the
    following system of 14 equations:
    \begin{eqnarray}
      &\xi^2_{xu}  =  \xi^2_{x}  =  \xi^2_{u}  = \xi^2_{uu} = 0,
      \label{eq:deteq1a}
      \\
      &\eta_{xt} + u \eta_{uu} = \eta_{xu} - \xi^2_{xt} - u\xi^2_{xx}
      = \xi^1_u - u\xi^2_u = \xi^1_u - 2u \xi^2_u = ru\xi^1_{uu} - \xi^1_u = 0,
      \\
      &-\xi^1_u+u\qp{\xi^2_u+r\qp{\xi^1_{uu}+u\xi^2_{uu}}}
      =
      \eta - \xi^1_t + u\qp{\xi^2_t-\xi^1_x+2u\xi^2_x}
      =
      r\eta_{uu} - r\xi^2_{tu} - 2\xi^2_{x}-r\xi^1_{xu}-2ru\xi^2_{xu}
      =0,
      \\
      &
      r\eta_{tu} + 2\eta_x + r\qp{2u\eta_{xu}-\xi^1_{xt}-u\xi^1_{xx}}
      =
      -\eta + u\eta_u+ru^2\eta_{uu}+\xi^1_t-ru\xi^1_{tu}-2ru^2\xi^1_{xu}
      =
      0.
      \label{eq:deteq2a}    
    \end{eqnarray}
    Solutions of the overdetermined system of linear PDEs
    \eqref{eq:deteq1a}-\eqref{eq:deteq2a} will yield the algebra of
    the symmetry generators \eqref{eq:vf} of the $r$-HS equation.
  \end{proposition}
  
  Given \eqref{eq:deteq1a}-\eqref{eq:deteq2a} form an overdetermined
  system of linear partial differential equations it is possible that
  they only admit the trivial solution $\xi^1=\xi^2=\eta=0$. This
  would imply that the only Lie symmetry of the $r$-HS equation is the
  identity transformation. In what follows we will see that this is
  not the case. We are able to obtain the Lie algebra for the symmetry
  generators the $r$-HS equation and thus, using the Lie series,
  derive the groups of Lie point symmetries.
  
  \begin{proposition}[Determining equations]
    The general solution of the determining equations
    \eqref{eq:deteq1a}-\eqref{eq:deteq2a} is given by
    \begin{equation}
      \label{eq:sol-gena}
      \xi^1= c_1 t + c_2 xt + c_3 + c_4 x, 
      \quad
      \xi^2=\frac{c_2t^2}r + c_4 t - c_5 t + c_6,
      \quad
      \eta=c_1 + c_2 \qp{tu - \frac{2tu}r + x} + c_5 u, 
    \end{equation}
    where $c_i$, $i=1,\dots, 6$ are arbitrary real constants. 
  \end{proposition}

  \begin{proposition}[Lie algebra generators]  
    \label{pro:Lie-algebra-gen}
    It follows that the solution \eqref{eq:sol-gena} defines a six
    dimensional Lie algebra of generators where a basis is formed by
    the following vector fields
    \begin{eqnarray}
      &X_1
      =
      \frac{\partial }{\partial x},
      \quad
      X_2
      =
      \frac{\partial }{\partial t},
      \quad
      X_3
      =
      \frac{\partial }{\partial u}
      +
      t\frac{\partial }{\partial x},
      \quad
      X_4
      =
      t\frac{\partial }{\partial t}
      +
      x\frac{\partial }{\partial x},
      \quad
      \\
      &X_5
      =
      u\frac{\partial }{\partial u}
      -
      t\frac{\partial }{\partial t},
      \quad
      X_6
      =
      \frac{t^2}r\frac{\partial }{\partial t}
      +
      \qp{tu - \frac{2tu}{r} + x}\frac{\partial }{\partial u}
      +
      tx\frac{\partial }{\partial x}.
    \end{eqnarray}
  \end{proposition}

\subsection*{Invariant solutions through symmetry reductions}

We now state solutions that occur through symmetry reductions of
(\ref{eq:rHS-Cauchy}) to ODEs by means of the algebra generators given
in Proposition \ref{pro:Lie-algebra-gen}. We consider each generator
seperately and examine some examples of solutions from each.

\subsection*{$X_1$}
To begin notice that solutions of (\ref{eq:rHS-Cauchy}) that are
invariant under the symmetry generated by $X_1$ are of the form $u =
f(t)$, which immediately yields $u\equiv \text{const}$ as a trivial
solution prescribed by the initial condition.

\subsection*{$X_2$}
Solutions invariant under the symmetry generated by $X_2$ are of the
form $u = f(x)$. The reduced equation is given by the following ODE:
\begin{equation}
  \qp{f_x}^r\qp{f_x^2 + rff_{xx}} = 0.
\end{equation}
This, in turn shows that either $f\equiv \text{const}$ or $f_x^2 +
rff_{xx}=0$ and must take the form
\begin{equation}
  f(x) = c_2\qp{x+rx-rc_1}^{r/(1+r)},
\end{equation}
for arbitrary constants $c_1,c_2$.

\subsection*{$X_3$} Solutions invariant under the symmetry generated by $X_3$ are of the form $u = x t^{-1} + f(t)$ with $f$ prescribed by the initial condition.

\subsection*{$X_4$}
The quantities $u$ and $\xi = tx^{-1}$ are algebraic invariants of the
Lie group generated by $X_4$. Assume $u=f(\xi)$ for non-constant $f$,
then we obtain the reduced equation as the following ODE:
\begin{equation}
  rf'(\xi)\qp{2\xi f(\xi) - 1} + \xi^2f'(\xi)^2 + r\xi f''(\xi)\qp{\xi f(\xi) - 1} = 0.
\end{equation}
The general solution of this ODE is not known.

\subsection*{$X_5$}
Solutions invariant under the symmetry generated by $X_5$ are of the
form $u = f(x) t^{-1}$, for non-constant $f$. The reduced ODE is given by:
\begin{equation}
  rf'(x) - f'(x)^2 - rf(x)f''(x) = 0,
\end{equation}
whose solution is an inverse hypergeometric function, that is
\begin{equation}
  f^{-1}(x) = \frac{_2F_1\qp{1,-r,1-r,\frac{c_1\qp{x+c_2}^{-1/r}}{r}}\qp{x+c_2}}{r},
\end{equation}
for constants $c_1,c_2$.

\subsection*{$X_6$}
The quantities $u$ and $\xi = tx^{-1/r}$ are algebraic invariants of
the Lie group generated by $X_6$. With $u=f(\xi)$ we may derive the
following reduced ODE:
\begin{equation}
  -\xi f(\xi ) \qp{ \qp{ -r^2+3 r+4 }
    f'(\xi) + \xi r f''(\xi ) } + \qp{r^2-4} f(\xi)^2-\xi ^2
  f'(\xi)^2=0,
\end{equation}
which has closed form solution
\begin{equation}
  f(\xi) = c_2 \exp\qp{
    \frac{2 r \log
      \qp{c_1+\xi^r}
      +
      \qp{-2 r-4} \log (\xi )}
         {2 \qp{r+1}}
  }.
\end{equation}

\section{Piecewise linear solutions}
\label{sec:linear}
In this section we consider an optimal control formulation that allows
us to quantify piecewise linear solutions to the problem. We show that
the computation of these solutions requires the solution of a
nonlinear system and make various comparisons to the $2$-HS
equation. We leave the formal interpretation of these solutions to the
next section.

\begin{definition}[Optimal control problem]
  \label{def:optimal}
  Let $Q_1(t),\ldots,Q_N(t)$ represent a moving set of points on the
  interval $[a,b]$. The $\sob{1}{r}$ optimal control problem is to
  find $Q_1(t),\ldots,Q_N(t)$ and $u\in \sob{1}{r}(a,b)$ such that
  \begin{equation}
    \int_0^T\int_a^b \frac{1}{r}\norm{u_x}^r\diff x \diff t
  \end{equation}
  is minimised, subject to the constraints
  \begin{equation}
    \begin{split}
      \dot{Q}_i(t) &= u(Q_i(t),t),
      \\
      Q_i(0)&=Q^A_i,
      \\
      Q_i(T)&=Q^B_i,
    \end{split}
  \end{equation}
  for $i=1,\ldots, N$.
\end{definition}
This problem has the following variational formulation.
\begin{definition}[Clebsch variational principle]
Let $Q_1(t),\ldots,Q_N(t)$ be a set of points on the interval $[a,b]$
with Lagrange multipliers $P_1,\ldots,P_N$, and suppose $u\in
\sob{1}{r}(a,b)$. The Clebsch variational principle corresponding to
the optimal control problem in Definition \ref{def:optimal} is
\begin{equation}
  \label{eq:clebsch1}
  \delta S[u,P,Q] = 0, \quad S = \int_{t=0}^T l[u] + \sum_{i=1}^nP_i(\dot{Q}_i
  -u(Q_i))\diff t, \, Q_i(0)=Q^A_i,\, Q_i(T)=Q^B_i,\,i=1,\ldots, N.
\end{equation}
\end{definition}
This type of variational principle takes its name from variational
principles of this form that can be used to derive equations of fluid
dynamics\revised{, where the Lagrange multipliers (Clebsch variables) enforce
the dynamics of transported quantities.}
\revised{\begin{remark}
In \cite{cotter2009continuous,gay2011clebsch} these variational
principles were considered in a general form {(to which the name
``Clebsch'' was extended)} where the Lagrange multipliers enforce dynamics
\begin{equation}
  \dot{Q} = L_uQ,
\end{equation}
where $Q$ is {a curve on a} manifold $\mathcal{M}$, $u$
{is a curve} on a Lie algebra $\mathcal{X}$, and $L_u$
represents a Lie algebra action on $\mathcal{M}$. In this case, it was
shown that $P$ and $Q$ can be eliminated {via the closure of the
corresponding Lie algebra bracket}. Then, $u$ solves the corresponding
Euler-Poincar\'e equation. In our case, we have $(L_uQ)_i = u(Q_i,t)$,
which is a Lie algebra action of smooth vector fields on $[a,b]$,
which have the Lie bracket $[u,v]=u_xv-uv_x$. However,
$\sob{1}{r}(a,b)$ is not closed under this bracket, and we see that we
do not have enough regularity to complete this argument. However, as
we shall find, {it is still possible to understand the evolution of $u$
as a weak solution of the corresponding Euler-Poincar\'e equation (the
$r$-Hunter-Saxton equation, in this case).}
\end{remark}}
\begin{lemma}
  The optimal $u$ to the variational problem (\ref{eq:clebsch1}) are
  given by piecewise linear functions, whose jumps occuring at $x=Q_i$,
  $i=1,\ldots,N$, and jump condition
  \begin{equation}
    \label{eq:jump}
  -\left[\norm{u_x}^{r-2}u_x\right]_{Q_i^-}^{Q_i^+} = P_i.
\end{equation}  
\end{lemma}
\begin{proof}
The Euler-Lagrange equation corresponding to $u$ is given by
\begin{equation}
  \int_0^1 \norm{u_x}^{r-2} u_x \delta u_x \diff x = \sum_{i=1}^n P_i\delta u(Q_i),
\end{equation}
which is, in turn, the weak form of the $r$-Laplace equation
\begin{equation}
\label{eq:mm}
  -\qp{\norm{u_x}^{r-2} u_x}_x = \sum_{i=1}^nP_i\delta(x-Q_i),
\end{equation}
where $\delta(x)$ is the Dirac measure.  This means that $u$ is
piecewise linear, with jump conditions given by \eqref{eq:jump}.
\end{proof}
Notice that this means that the term $(\norm{u_x}^{r-2}u_x)_xu_x$ in
\eqref{eq:m_t} does not make sense, since at $x=Q_i$ it is the product
of a Dirac measure with a function that has a jump at the same
location. Later we shall see that this problem is resolved since $u$
solves \eqref{eq:weak rHS} nevertheless.
\begin{remark}
  Equation \eqref{eq:mm} defines an infinitesimally equivariant momentum
  map from $(P,Q)\in T^*([a,b]^n)$ to the dual space
  $\mathcal{X}([a,b])$ containing the momentum
  $m=-\qp{\norm{u_x}^{r-2}u_x}_x$. The equivariance explains why it is
  possible to eliminate $P$ and $Q$ and obtain an equation for $u$
  alone. This was explored in the context of peakon solutions of the
  Camassa-Holm equation in \cite{holm2005momentum}, where the
  infinitesimal equivariance is proved.
\end{remark}

\begin{lemma}
  The piecewise solution can be characterised through the points
  $\widehat{u}_i = u(Q_i)$, $i=1,\ldots n$ with $\widehat{u}_0=0$,
  $\widehat{u}_{n+1}=0$. Then, the set of coefficients
  $\{\widehat{u}_i\}_{i=0}^{n+1}$ solves the difference equation
  \begin{equation}
    \label{eq:U}
    \begin{split}
      P_i
      &=
      -\norm{\frac{\widehat{u}_{i+1}-\widehat{u}_i}{Q_{i+1}-Q_i}}^{r-2}
      \qp{\frac{\widehat{u}_{i+1}-\widehat{u}_i}{Q_{i+1}-Q_i}}
      +
      \norm{\frac{\widehat{u}_{i}-\widehat{u}_{i-1}}{Q_{i}-Q_{i-1}}}^{r-2}
      \qp{\frac{\widehat{u}_{i}-\widehat{u}_{i-1}}{Q_{i}-Q_{i-1}}}
      \text{ for } i=1,\ldots,n,
      \\
      \widehat{u}_0 &=
      \widehat{u}_{n+1}=0.
    \end{split}
  \end{equation}  
\end{lemma}
\begin{proof}
The piecewise linear solution has piecewise constant derivative,
\begin{equation}
u_x = \frac{\widehat{u}_{i+1}-\widehat{u}_i}{Q_{i+1}-Q_i},
\quad i=1,\ldots,n,
\end{equation}
and substitution into \eqref{eq:jump} gives the result.
\end{proof}
\begin{lemma}
  Let $Q_{i+1}>Q_i$ for $i=0,\ldots,n$. Then, equation \eqref{eq:U}
  has a unique solution.
\end{lemma}
\begin{proof}
  Let
  \begin{equation}
    \widehat{l}(\widehat{u},Q)
    =
    \sum_{i=0}^n
    \frac{\norm{\widehat{u}_{i+1}-\widehat{u}_i}^r}
         {r\norm{Q_{i+1}-Q_i}^{r-2}\qp{Q_{i+1}-Q_i}}.
  \end{equation}
  The solution $\{\widehat{u}_i\}_{i=0}^{n+1}$ is then the minimiser of 
  \begin{equation}
    \widehat{l}(\widehat{u},Q)- \sum_{i=1}^n P_i\widehat{u}_i.
  \end{equation}
  Writing $\Delta\widehat{u}_i= \widehat u_{i+1} - \widehat u_i$, $i=0,\dots,n$, we see that
  $\{\Delta\widehat{u_i}\}_{i=0}^{n+1}$ minimises the function
\begin{equation}
\sum_{i=0}^n
\frac{\norm{\Delta\widehat{u}_i}^r}
     {r\norm{Q_{i+1}-Q_i}^{r-2}\qp{Q_{i+1}-Q_i}}
  - \sum_{i=1}^n P_i \sum_{j=0}^i(\Delta\widehat{u}_i),
\end{equation}
subject to the constraint $\sum_{i=1}^n\Delta\widehat{u}_i=0$. The
function is the sum of convex and linear functions, which is then also
convex. Since the constraint defines a linear subspace, we are
minimising a convex function over a finite dimensional vector space,
which has a unique solution.
\end{proof}

The solution does not have a closed form analytic expression but can
be solved numerically using Newton's method. We denote the solution
operator $\widehat{u}=U(P,Q)$.

\begin{theorem}
The equations for $P$ and $Q$ satisfy
\begin{align}
  \label{eq:Qdot}
  \dot{Q}_i & = \widehat{u}_i, \\
  \dot{P}_i
  &=
  \frac{r-1}{r}\qp{
    \norm{
      \frac{\widehat{u}_i-\widehat{u}_{i-1}}{Q_i-Q_{i-1}}
    }^r
    -
    \norm{
      \frac{\widehat{u}_{i+1}-\widehat{u}_{i}}{Q_{i+1}-Q_{i}}
    }^r
  }, \label{eq:Pdot} \\
\widehat{u} & = U(P,Q),
\label{eq:hatu}
\end{align}
with the conserved energy
\begin{equation}
  E[P,Q] = \widehat{l}(U(P,Q),Q).
\end{equation}
\end{theorem}
\begin{proof}
  We know from \cite{cotter2009continuous,gay2011clebsch} that the
  equations for $P$ and $Q$ are canonical Hamiltonian with Hamiltonian
  function given by Legendre transform
  \begin{equation}
    H(P,Q)
    =
    \sum_j P_j\widehat{u}_j - \widehat{l}(\widehat{u},Q),
    \quad
    \widehat{u} = U(P,Q),
  \end{equation}
since $\widehat{l}(\widehat{u},Q)$ is equal to $l(u)$ when $u$ is the piecewise
linear function interpolating $\widehat{u}$.
The equation for $Q$ is then
\begin{equation}
  \dot{Q}_i = \pp{}{P_i}H(P,Q) = \widehat{u}_i + \sum_j
  \underbrace{\left(P_j - \pp{}{\widehat{u}_j}\widehat{l}(\widehat{u},Q)\right)}_{=0}
  \pp{}{P_i}U_j(P,Q) = \widehat{u}_i,
\end{equation}
as expected, after noting that the bracket vanishes since $\widehat{u}$
satisfies \eqref{eq:U}.
After noting that
\begin{equation}
  \pp{}{Q_i} \widehat{l}(\widehat{u},Q)
  =
  -\frac{r-1}{r}
  \qp{
        \norm{
      \frac{\widehat{u}_i-\widehat{u}_{i-1}}{Q_i-Q_{i-1}}
    }^r
    -
    \norm{
      \frac{\widehat{u}_{i+1}-\widehat{u}_{i}}{Q_{i+1}-Q_{i}}
    }^r
  },
\end{equation}
we see the equation for $P$ is given by
\begin{align}
  \dot{P}_i &= -\pp{H}{Q_i},\\
  &=  \sum_j
  \underbrace{\left(P_j-\pp{}{\widehat{u}_j}\widehat{l}(\widehat{u},Q)\right)}_{=0}
  \pp{}{Q_i}U_j(P,Q) - \pp{}{Q_i}\widehat{l}(\widehat{u},Q),\\
  &=
  \frac{r-1}{r}
  \qp{
        \norm{
      \frac{\widehat{u}_i-\widehat{u}_{i-1}}{Q_i-Q_{i-1}}
    }^r
    -
    \norm{
      \frac{\widehat{u}_{i+1}-\widehat{u}_{i}}{Q_{i+1}-Q_{i}}
    }^r
  },
\end{align}
as required.
\end{proof}
Hence, solving the nonlinear equation \eqref{eq:U} for $\widehat{u}$
allows us to calculate both $\dot{P}$ and $\dot{Q}$ and numerically
integrate the equations.

Note that when $r=2$, the $P$ equation specialises to
\begin{align}
  \dot{P}_i &=   \frac{1}{2}\left(
  \frac{\widehat{u}_i-\widehat{u}_{i-1}}{Q_i-Q_{i-1}}
  \right)^2-\frac{1}{2}
  \left(
  \frac{\widehat{u}_{i+1}-\widehat{u}_{i}}{Q_{i+1}-Q_{i}}
  \right)^2, \\
  &= \frac{1}{2}\left(-\left(
  \frac{\widehat{u}_{i+1}-\widehat{u}_{i}}{Q_{i+1}-Q_{i}}
  \right)
    +\left(
  \frac{\widehat{u}_i-\widehat{u}_{i-1}}{Q_i-Q_{i-1}}
  \right)\right)
    \left(\left(
  \frac{\widehat{u}_{i+1}-\widehat{u}_{i}}{Q_{i+1}-Q_{i}}
  \right)
    +\left(
  \frac{\widehat{u}_i-\widehat{u}_{i-1}}{Q_i-Q_{i-1}}
  \right)\right), \\
  & = -P_i\frac{1}{2}\left(u_x|_{Q_i^-}+u_x|_{Q_i^+}\right),
\end{align}
which recovers an equation from the standard ($r=2$) Hunter-Saxton case.
\begin{remark}
In the $r=2$ case, this formula allows an \emph{ad hoc} interpretation
of $mu_x$ term applied the piecewise linear solutions, when $m$ contains
Dirac measures centred on points where $u_x$ jumps. This interpretation
says that we can just take an average of $u_x$ from each side of the jump
point, but this only coincidentally works for the $r=2$ case, and is
\emph{not true} for $r>2$.
\end{remark}
\begin{remark}
The same $(P,Q)$ dynamics is obtained from the
following variational principle,
\begin{equation}
  \delta \widehat{S}(\widehat{u},P,Q)=0, \quad \widehat{S} = \int_0^T
  \widehat{l}(\widehat{u},Q) + \sum_{i=1}^n P_i\cdot \left(\dot{Q}_i -
  \widehat{u}_i\right), \, Q_i(0)=Q^A_i,\, Q_i(T)=Q^B_i,\,i=1,\ldots, N.
\end{equation}
This dynamics has a $Q$-dependent Lagrangian but trivial Lie algebra
action of $\widehat{u}$ on $Q$. The equivalence of these two formulations
becomes clear after noticing that they result in the same Hamiltonian
after Legendre transformation. However this second variational principle
has more directly computable Euler-Lagrange equations, since
the vanishing terms above simply do not appear in the first place.
\end{remark}

Finally, we examine the $r\to \infty$ limit of Equations
(\ref{eq:Qdot}-\ref{eq:hatu}).
\begin{corollary}
  \label{thm:rinfty}
  \revised{Let $z_i=P^{\frac{1}{r-1}}$}. Then, the $r\to \infty$ limit of
  Equations (\ref{eq:Qdot}-\ref{eq:hatu}) are given by
  \begin{align}
    \label{eq:Qz}
    \dot{Q}_i & = \hat{u}_i, \\
    \label{eq:uz}
    \argmax\left(\revised{-}\frac{\hat{u}_{i+1}-\hat{u}_i}{Q_{i+1}-Q_i},
    \frac{\hat{u}_i-\hat{u}_{i-1}}{Q_i-Q_{i-1}}\right)
    & = z_i, \\
    \revised{\dot{z}_i} &\revised{ = 0.}
  \end{align}
\end{corollary}
This limit makes sense, since it keeps $\hat{u}_i$ finite
(for distinct $Q_i$).
\begin{proof}
  First we examine the difference equation \eqref{eq:U}. After
  substituting $z_i$ into $P_i$ and taking the $r-1$ root, we obtain
  \begin{equation}
    z_i =
    \left(-\norm{\frac{\widehat{u}_{i+1}-\widehat{u}_i}{Q_{i+1}-Q_i}}^{r-2}
      \qp{\frac{\widehat{u}_{i+1}-\widehat{u}_i}{Q_{i+1}-Q_i}}
      +
      \norm{\frac{\widehat{u}_{i}-\widehat{u}_{i-1}}{Q_{i}-Q_{i-1}}}^{r-2}
      \qp{\frac{\widehat{u}_{i}-\widehat{u}_{i-1}}{Q_{i}-Q_{i-1}}}
\right)^{\frac{1}{r-1}}
  \end{equation}
  which recovers \eqref{eq:uz} in the limit. Substituting $z_i$ into
  Equation \eqref{eq:Pdot} gives
  \begin{equation}
    (r-1)z_i^{r-2}\dot{z}_i
    =
  \left(\frac{r-1}{r}\qp{
    \norm{
      \frac{\widehat{u}_i-\widehat{u}_{i-1}}{Q_i-Q_{i-1}}
    }^r
    -
    \norm{
      \frac{\widehat{u}_{i+1}-\widehat{u}_{i}}{Q_{i+1}-Q_{i}}
    }^r
  }\right),
  \end{equation}
  and hence
  \begin{align}
    \dot{z}_i
    & =
    \frac{1}{(r-1)z_i^{r-2}}\left(\frac{r-1}{r}\qp{
    \norm{
      \frac{\widehat{u}_i-\widehat{u}_{i-1}}{Q_i-Q_{i-1}}
    }^r
    -
    \norm{
      \frac{\widehat{u}_{i+1}-\widehat{u}_{i}}{Q_{i+1}-Q_{i}}
    }^r
    }\right), \\
        & =
    \frac{1}{(r-1)P_i^{\frac{r-2}{r-1}}}\left(\frac{r-1}{r}\qp{
    \norm{
      \frac{\widehat{u}_i-\widehat{u}_{i-1}}{Q_i-Q_{i-1}}
    }^r
    -
    \norm{
      \frac{\widehat{u}_{i+1}-\widehat{u}_{i}}{Q_{i+1}-Q_{i}}
    }^r
  }\right), \\
        & \revised{=
    \frac{1}{r}
    \frac{\qp{
    \norm{
      \frac{\widehat{u}_i-\widehat{u}_{i-1}}{Q_i-Q_{i-1}}
    }^r
    -
    \norm{
      \frac{\widehat{u}_{i+1}-\widehat{u}_{i}}{Q_{i+1}-Q_{i}}
    }^r
  }}{\left(      -\norm{\frac{\widehat{u}_{i+1}-\widehat{u}_i}{Q_{i+1}-Q_i}}^{r-2}
      \qp{\frac{\widehat{u}_{i+1}-\widehat{u}_i}{Q_{i+1}-Q_i}}
      +
      \norm{\frac{\widehat{u}_{i}-\widehat{u}_{i-1}}{Q_{i}-Q_{i-1}}}^{r-2}
      \qp{\frac{\widehat{u}_{i}-\widehat{u}_{i-1}}{Q_{i}-Q_{i-1}}}
\right)^\frac{r-2}{r-1}}, }
  \end{align}  
  which converges to zero as $r\to \infty$, provided that
  $(\hat{u}_i-\hat{u}_{i-1})/(Q_i-Q_{i-1})$ is bounded for
  all $i$.
\end{proof}
The well-posedness of Equations (\ref{eq:Qdot}-\ref{eq:hatu}) depends
on the solveability of \eqref{eq:hatu}, which \revised{is} a linear system
in the max-algebra (the algebra where the addition operator is
replaced by the max operator) that can be formulated as a linear program.
The convergence of solutions of Equations (\ref{eq:Qdot}-\ref{eq:hatu})
to solutions of Equations (\ref{eq:Qz}-\ref{eq:uz}\revised{)} is an open question,
which we investigate through special cases and numerical solutions
in Section \ref{sec:numerics}.
\section{Weak solutions of the r-HS equation}
\label{sec:weak}
As previously discussed, it is natural to hope the results of
\cite{cotter2009continuous,gay2011clebsch} allow us to show that the
piecewise linear $u$ derived in \S \ref{sec:linear} satisfies the
corresponding Euler-Poincar\'e equation, \emph{i.e.,} the
r-Hunter-Saxton equation \eqref{eq:m_t}. However, since the
corresponding Lie bracket is not closed in $\sob{1}{r}$, it is only
closed in $C^\infty$, the results of those papers do not hold. Indeed,
\eqref{eq:m_t} is not well-defined for piecewise linear solutions, not
even weakly, since the middle integral \revised{in \eqref{eq:EP weak},
\begin{equation}
  \int_0^1 \qp{\norm{u_x}^{r-2}u_x}_xwu_x\diff x,
\end{equation}
is not defined for test functions $w \in H^1([0,1])$, even after integrating
by parts.} However, remarkably, the piecewise linear solutions
\emph{are} solutions of \eqref{eq:weak rHS}.
\begin{theorem}
  The piecewise linear functions evolving according to the Hamiltonian
  system above are solutions of equation \eqref{eq:weak rHS}.
\end{theorem}
\begin{proof}
  \revised{The proof is via straightforward direct calculation.}
  Let $u(x,t)$ be a time-dependent piecewise linear function over the
  intervals $[Q_i(t),Q_{i+1}(t)]$, with $u_i = u(Q_i(t))$.  Then,
  \begin{equation}
    \begin{split}
    \dd{}{t} \int_0^1 \norm{u_x}^{r-2} u_x \phi \diff x
    & =
    \dd{}{t}\sum_{i=1}^n \int_{Q_i}^{Q_{i+1}} \phi \diff x
    \norm{\frac{u_{i+1}-u_i}{Q_{i+1}-Q_i}}^{r-2}  \frac{u_{i+1}-u_i}{Q_{i+1}-Q_i}, \\
    & =
    \sum_{i=1}^n \int_{Q_i}^{Q_{i+1}} \phi \diff x
    \norm{\frac{u_{i+1}-u_i}{Q_{i+1}-Q_i}}^{r-2}  \frac{u_{i+1}-u_i}{Q_{i+1}-Q_i}
    \\
    &\qquad +
    \sum_{i=1}^n (\dot{Q}_{i+1}\phi(Q_{i+1})
    - \dot{Q}_i\phi(Q_i))
    \norm{\frac{u_{i+1}-u_i}{Q_{i+1}-Q_i}}^{r-2}  \frac{u_{i+1}-u_i}{Q_{i+1}-Q_i}
    .
    \end{split}
  \end{equation}
  Now, since $u$ is piecewise linear, $u_x$ is constant in every
  interval $[Q_i,Q_{i+1}]$, hence
  \begin{equation}
    \qp{u\norm{u_x}^{r-2}u_x}_x = \norm{u_x}^r,
  \end{equation}
  and, upon splitting the integral into intervals and integrating by parts,
  \begin{equation}
    \begin{split}
    -\int_0^1 u\norm{u_x}^{r-2}u_x  \phi_x\diff x
    &=
    -\sum_{i=0}^n \int_{Q_i}^{Q_{i+1}}  u \norm{u_x}^{r-2}u_x \phi_x \diff x,
    \\
    &=
    -\sum_{i=0}^n \qp{u_{i+1}\phi(Q_{i+1}) - u_i\phi(Q_i)}
    \norm{\frac{u_{i+1}-u_i}{Q_{i+1}-Q_i}}^{r-2}  \frac{u_{i+1}-u_i}{Q_{i+1}-Q_i}
    \\
    &\qquad +
    \sum_{i=1}^n \int_{Q_i}^{Q_{i+1}} \norm{u_x}^r  \phi \diff x,
    \\
    & =
    -\sum_{i=0}^n \qp{u_{i+1}\phi(Q_{i+1}) - u_i\phi(Q_i)}
    \norm{\frac{u_{i+1}-u_i}{Q_{i+1}-Q_i}}^{r-2}  \frac{u_{i+1}-u_i}{Q_{i+1}-Q_i}
    \\
    &\qquad +
    \sum_{i=1}^n \int_{Q_i}^{Q_{i+1}} \phi \diff x
    \norm{\frac{u_{i+1}-u_i}{Q_{i+1}-Q_i}}^{r}
    , 
    \end{split}
  \end{equation}
  Combining these together, we have
  \begin{equation}
    \begin{split}
      c(t)\sum_{i=1}^n
      \Phi_i
      &=
      \sum_{i=1}^n ((\dot{Q}_{i+1}-u_{i+1})\phi(Q_{i+1})
      - (\dot{Q}_i-u_i)\phi(Q_i))
      \norm{\frac{u_{i+1}-u_i}{Q_{i+1}-Q_i}}^{r-2}  \frac{u_{i+1}-u_i}{Q_{i+1}-Q_i}
      \\
      &\qquad +
      \sum_{i=1}^n
      \Phi_i \dd{}{t}
      \norm{\frac{u_{i+1}-u_i}{Q_{i+1}-Q_i}}^{r-2}  \frac{u_{i+1}-u_i}{Q_{i+1}-Q_i}
      +
      \sum_{i=1}^n \Phi_i \frac{r-1}{r}
      \norm{\frac{u_{i+1}-u_i}{Q_{i+1}-Q_i}}^{r}
      ,
    \end{split}
  \end{equation}
where
\begin{equation}
  \Phi_i = \int_{Q_i}^{Q_{i+1}}\phi \diff x, \quad i=1,\ldots, n.
\end{equation}
Taking $\dot{Q_i}=u_i$, we are left with
\begin{equation}
  \dd{}{t}
  \norm{\frac{u_{i+1}-u_i}{Q_{i+1}-Q_i}}^{r-2}  \frac{u_{i+1}-u_i}{Q_{i+1}-Q_i}
  +
  \frac{r-1}{r} \norm{\frac{u_{i+1}-u_i}{Q_{i+1}-Q_i}}^{r}
   = c(t),
\end{equation}
since $\Phi_i$ are arbitrary. Now we show that our singular
solutions also satisfy this equation.

First, define $c(t)$ according to
\begin{equation}
  c(t) =
  \dd{}{t}
  \norm{\frac{u_{1}-u_0}{Q_{1}-Q_0}}^{r-2}  \frac{u_{1}-u_0}{Q_{1}-Q_0}
  +
  \frac{r-1}{r} \norm{\frac{u_{1}-u_0}{Q_{1}-Q_0}}^{r}.
\end{equation}
Now combine equation \eqref{eq:U} and \eqref{eq:Pdot}, we have
\begin{equation}
  \begin{split}
    0&=
    \dd{}{t}\qp{
      -
      \norm{\frac{u_{i+1}-u_i}{Q_{i+1}-Q_i}}^{r-2}  \frac{u_{i+1}-u_i}{Q_{i+1}-Q_i}
      +
      \norm{\frac{u_{i}-u_{i-1}}{Q_{i}-Q_{i-1}}}^{r-2}  \frac{u_{i}-u_{i-1}}{Q_{i}-Q_{i-1}}
    }
    \\
    &\qquad +
    \frac{r-1}{r}
    \qp{
      -\norm{\frac{\widehat{u}_i-\widehat{u}_{i-1}}{Q_i-Q_{i-1}}}^r
      +
      \norm{\frac{\widehat{u}_{i+1}-\widehat{u}_{i}}{Q_{i+1}-Q_{i}}}^r
    },
  \end{split}
\end{equation}
which means that 
\begin{equation}
  \dd{}{t}
  \qp{
    \norm{\frac{u_{i+1}-u_i}{Q_{i+1}-Q_i}}^{r-2}  \frac{u_{i+1}-u_i}{Q_{i+1}-Q_i}
  }
  +
  \frac{r-1}{r}
  \norm{
    \frac{{u}_{i+1}-{u}_{i}}{Q_{i+1}-Q_{i}}
  }^r
  =
  \dd{}{t}
  \qp{
    \norm{\frac{u_{i}-u_{i-1}}{Q_{i}-Q_{i-1}}}^{r-2}  \frac{u_{i}-u_{i-1}}{Q_{i}-Q_{i-1}}
  }
  +
  \frac{r-1}{r}
  \norm{
    \frac{{u}_{i}-{u}_{i-1}}{Q_{i}-Q_{i-1}}
  }^r
  ,
\end{equation}
and hence
\begin{equation}
  \dd{}{t}
  \qp{
    \norm{\frac{u_{i+1}-u_i}{Q_{i+1}-Q_i}}^{r-2}  \frac{u_{i+1}-u_i}{Q_{i+1}-Q_i}
  }
  +
  \frac{r-1}{r}
  \norm{
    \frac{{u}_{i+1}-{u}_{i}}{Q_{i+1}-Q_{i}}
  }^r = c(t),
\end{equation}
by induction, and we have our required equation.
\end{proof}

\section{Numerical examples}
\label{sec:numerics}
In this section we compute a number of examples of piecewise linear
solutions, obtained by integrating equations
(\ref{eq:Qdot}-\ref{eq:hatu}).
\subsection{One point solutions}
In the case of a single point $Q_1$, we may use energy conservation to
find an algebraic dependence between $\widehat{u}_1$ and
$Q_1$. Specialising to $[a,b]=[0,1]$ to simplify, the conserved energy
is
\begin{equation}
  E =
  \frac{1}{r}\qp{
    Q_1 \qp{ \frac{\revised{\widehat{u}_1}}{Q_1} }^r
    +
    (1-Q_1)\qp{
      \frac{\revised{\widehat{u}_1}}{1-Q_1}}^r}  
 = 
  \frac{\revised{\widehat{u}^r_1}}{r}\qp{Q_1^{1-r} + (1-Q_1)^{1-r}}.
\end{equation}
Since $E$ is time-independent, we may compute it from the initial
condition, and then
\begin{equation}
  \revised{\widehat{u}_1} = \left(\frac{r E}{Q_1^{1-r} + (1-Q_1)^{1-r}}\right)^{\frac{1}{r}},
\end{equation}
assuming a positive root (it is a consequence of the one point $P$
equation that $u$ will stay positive if it is positive initially). We
may then integrate the first-order equation $\dot{Q_1}=\revised{\widehat{u}_1}(Q_1)$ (this
must be done numerically in general). Since $u$ is always positive,
$Q_1$ is monotonically increasing, and so it makes sense to consider what
happens when $Q_1\to 1$. Writing $Q_1=1-\epsilon$, we obtain the asymptotic
formula
\begin{equation}
  \revised{\widehat{u}_1} \sim E^{1/r}r^{1/r} \epsilon^{(r-1)/r}, \revised{\mbox{ as }\epsilon \to 0.}
\end{equation}
The equation
\begin{equation}
  \dd{\epsilon}{t} = -(E r)^{1/r}\epsilon^{(r-1)/r}
\end{equation}
has solution
\begin{equation}
  \epsilon = \qp{
    \epsilon^{1/r}_0
    -
    \frac{\qp{Er}^{1/r} t}{r}
  }^r,
\end{equation}
which converges to zero in finite time. At the same time, the
derivative of the piecewise linear function $u$ in the region $[Q_1,1]$ is given by
\begin{equation}
  u_x = \frac{-\revised{\widehat{u}_1}}{1-Q_1} \sim (E r)^{1/r}\epsilon^{-1/r}
    \to \infty \mbox{ as } \epsilon\to 0,
\end{equation}
indicating that the $\sob{1}{\infty}$-norm blows up in finite time.

Snapshots of a typical solution are shown in Figure \ref{fig:r2p1}.
A comparison of the behaviour of single point solutions is made
in Figure \ref{fig:r2-6p1}. We observe that the $\revised{\widehat{u}_1}-Q_1$ relation approaches
$\revised{\widehat{u}_1}=\min(Q_1,1-Q_1)$ as $r\to \infty$.
\begin{figure}
  \centerline{\includegraphics[width=9cm]{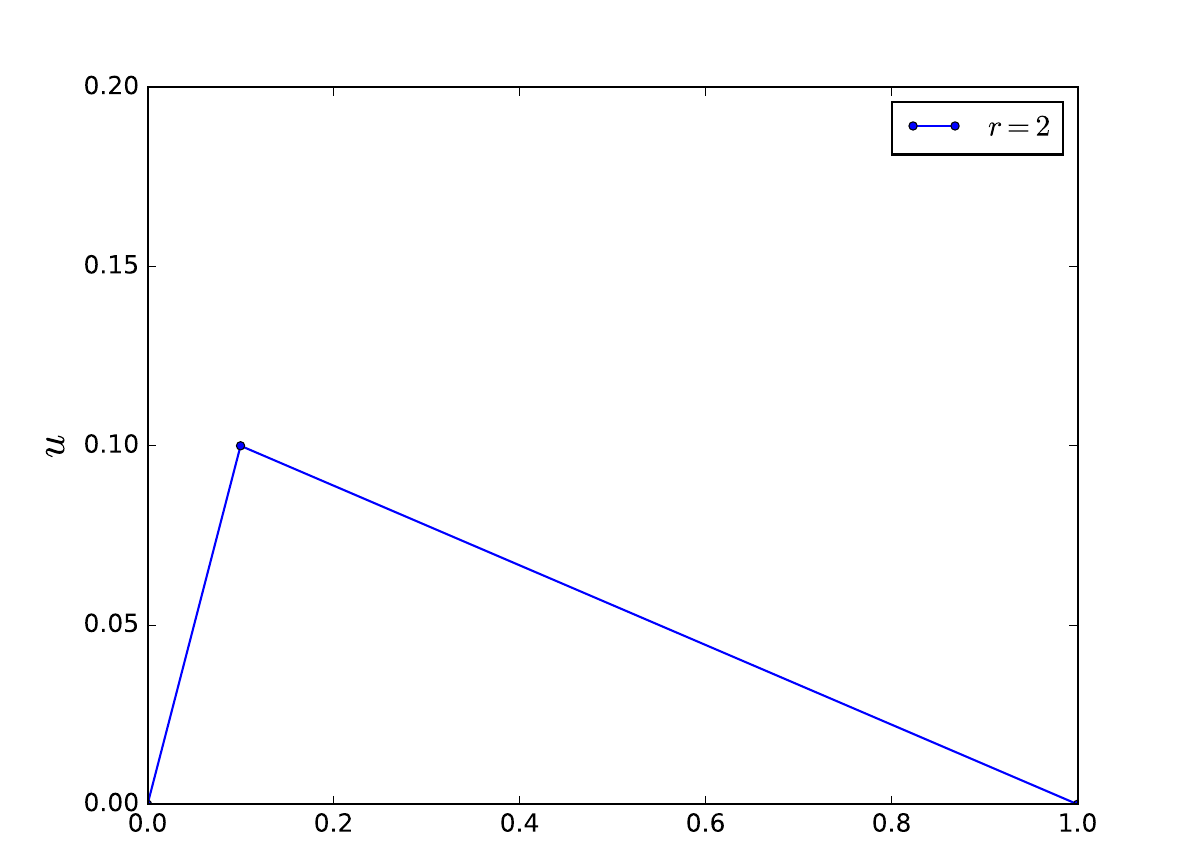}
    \includegraphics[width=9cm]{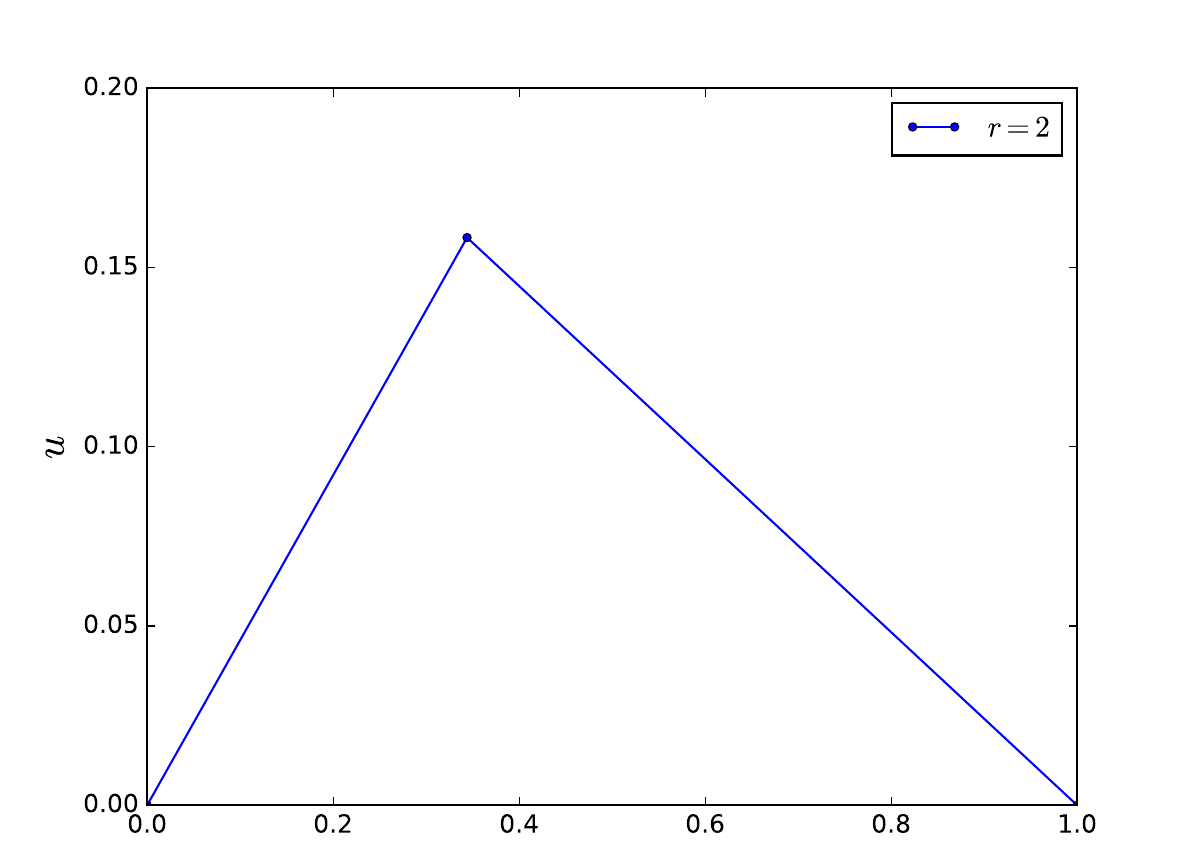}}
  \centerline{\small$x\hspace{90mm}x$}
  \centerline{\includegraphics[width=9cm]{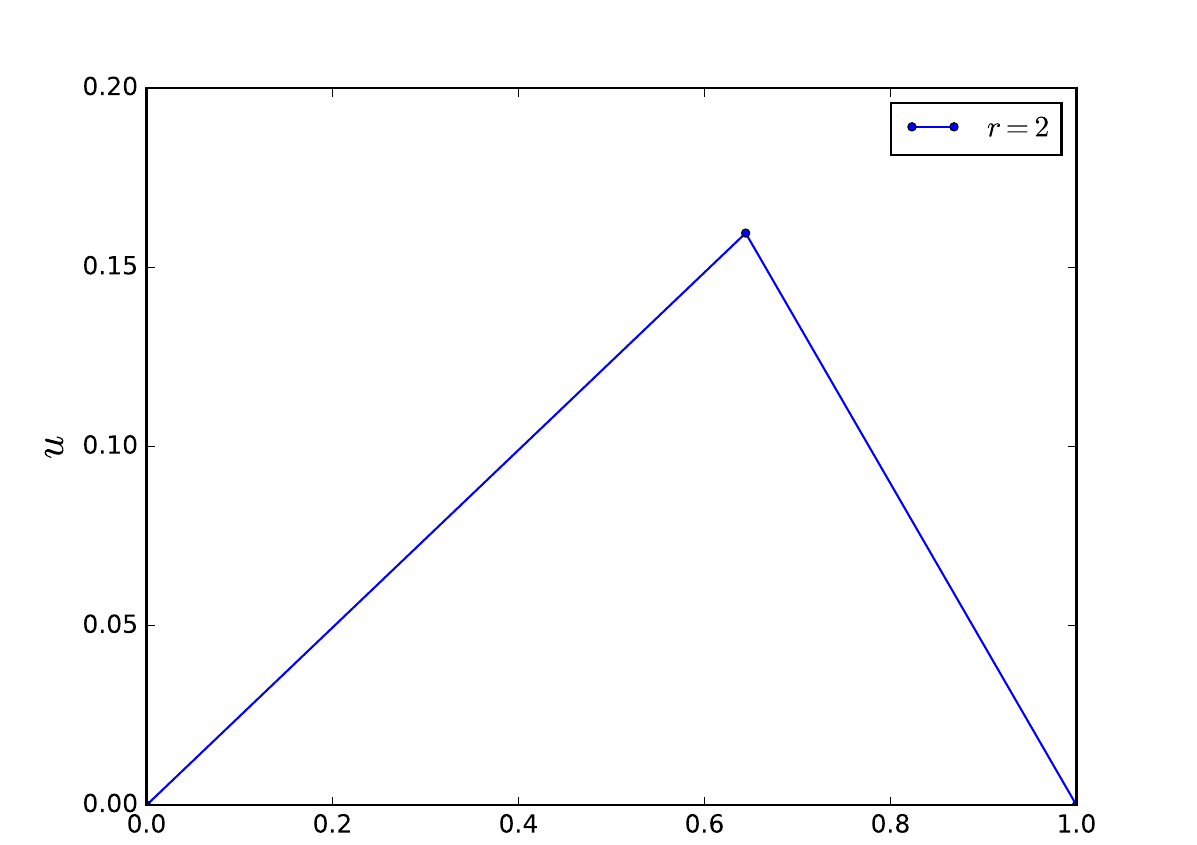}
    \includegraphics[width=9cm]{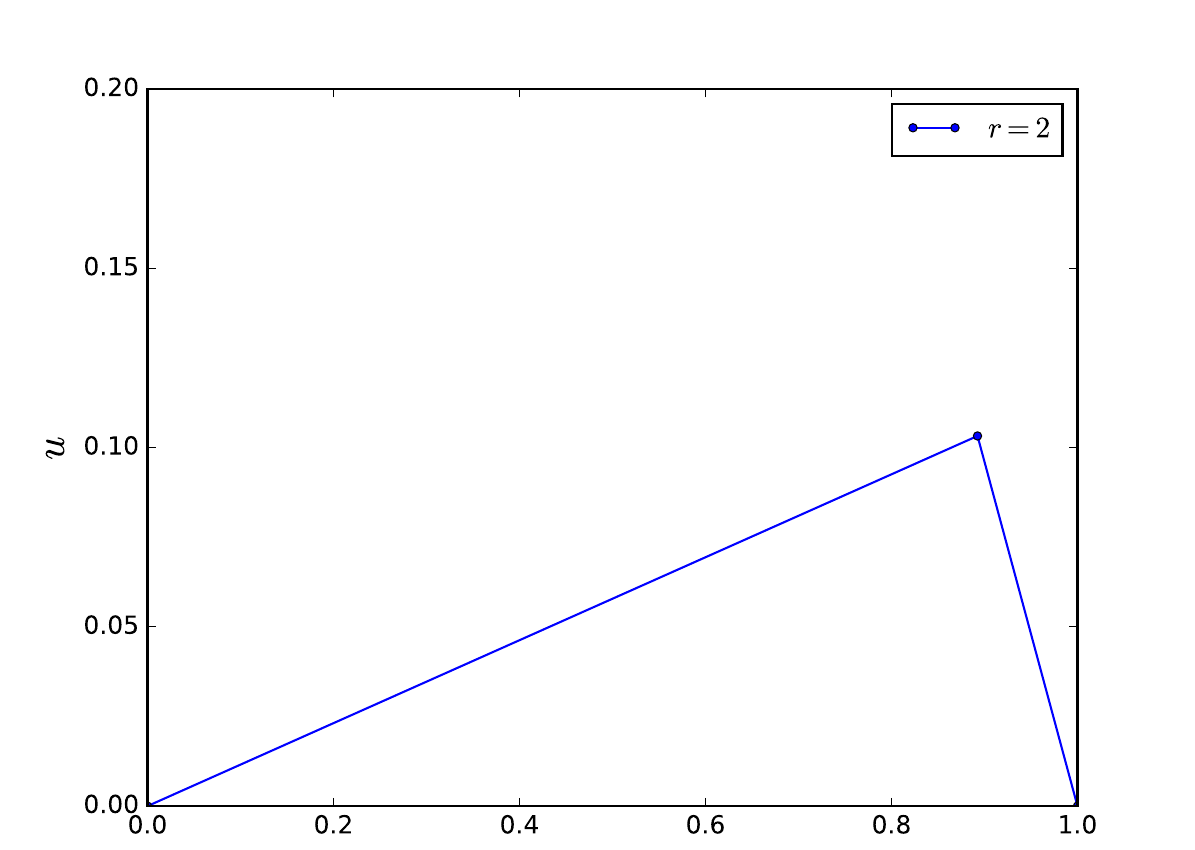}}
  \centerline{\small$x\hspace{90mm}x$}
  \caption{\label{fig:r2p1}Plots of $u(x)$ for a one point solution
    with $r=2$, at times $t=0.,1.83,3.67,5.5$.}
\end{figure}
\begin{figure}
    \centerline{\includegraphics[width=9cm]{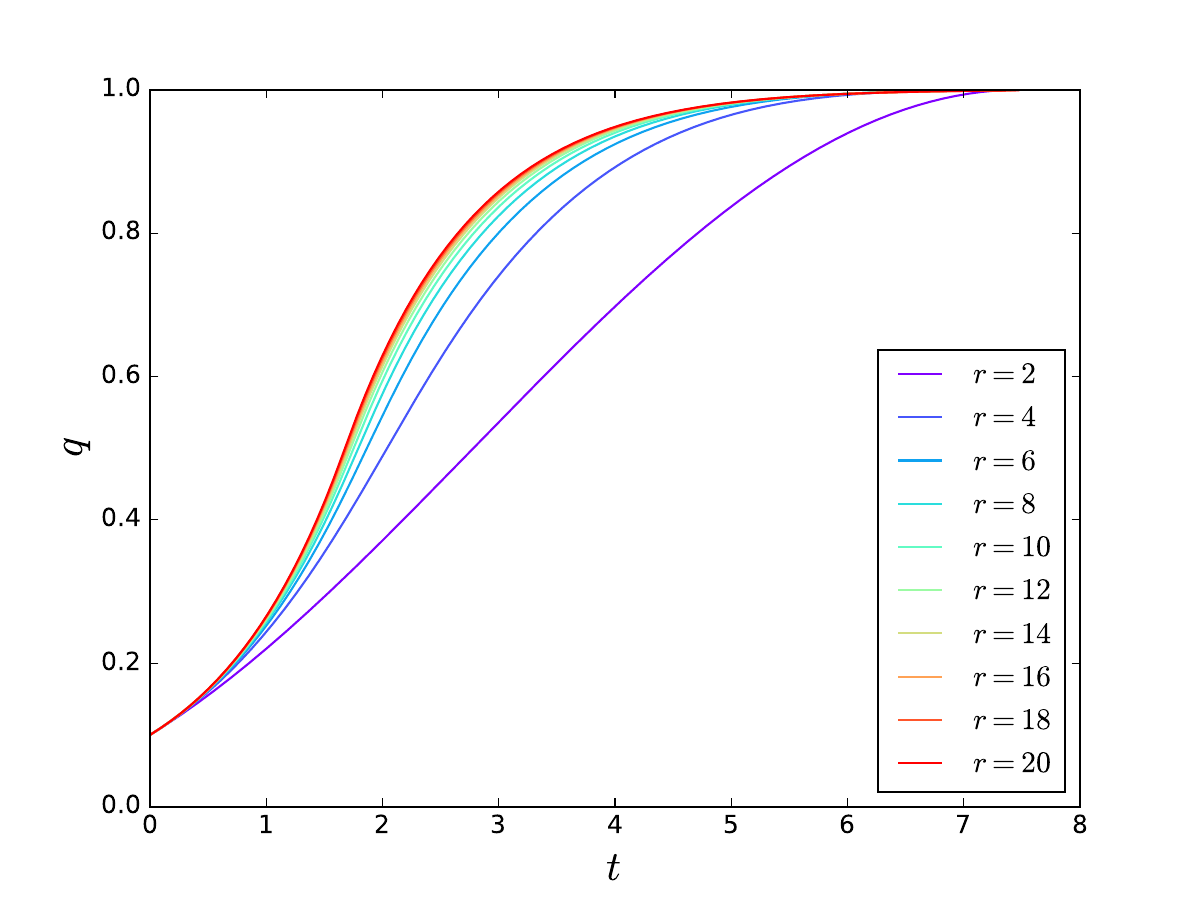}
    \includegraphics[width=9cm]{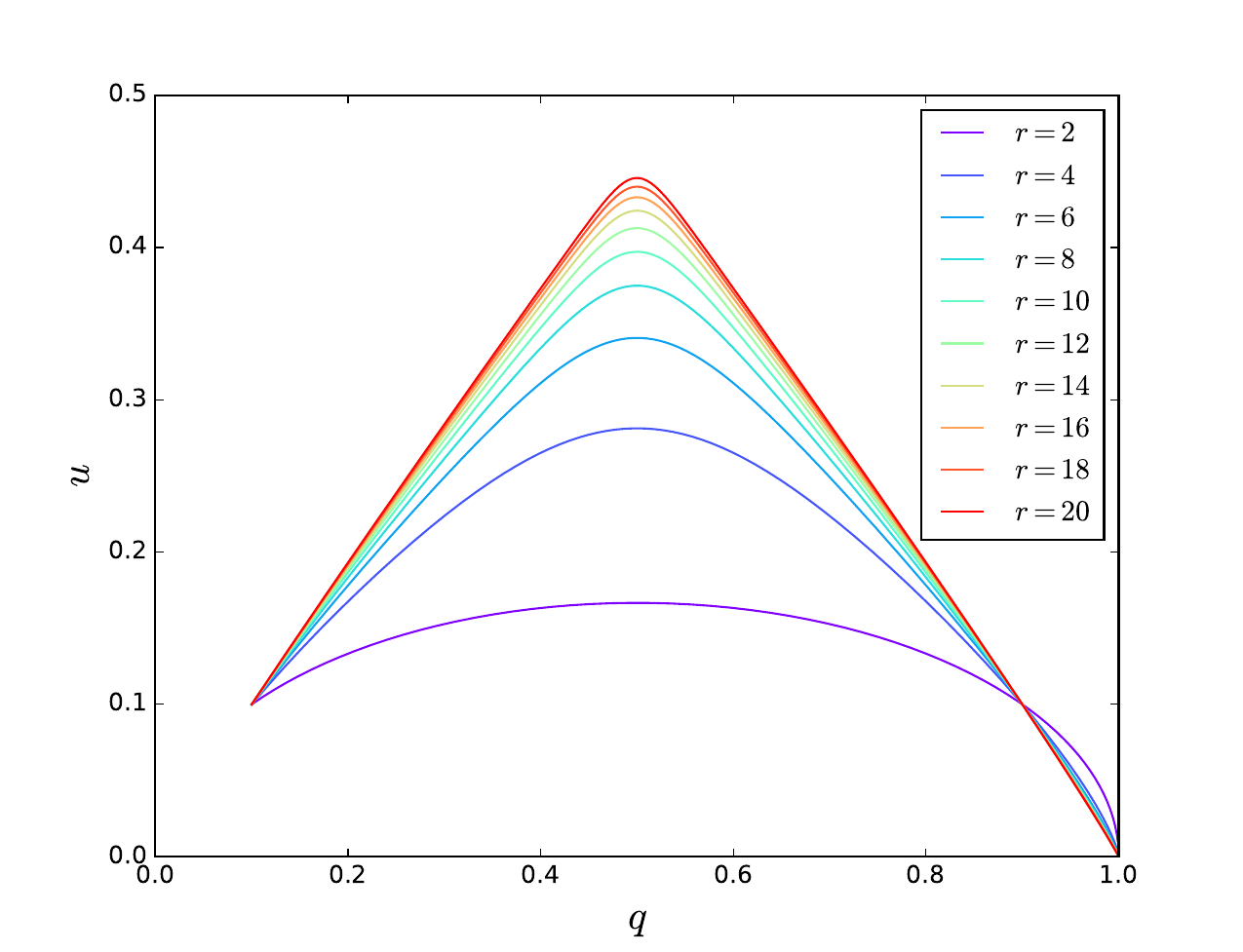}}
    \caption{\label{fig:r2-6p1}Comparison of one point trajectories
      for different values of $r$, 2, 4, 6, 8, 10, 12, 14, 16, 18, 20, with
      initial \revised{condition} $\revised{(Q_1,\widehat{u}_1)}=(0.1,0.1)$. Left: the point
      locations $\revised{q=Q_1(t)}$ are plotted against time. Right: the peak velocity
      $\revised{\widehat{u}_1}$ is plotted against $\revised{q=Q_1(t)}$.}
\end{figure}
A plot of the derivative $u_x$ in $[Q_1,1]$ is given in Figure
\ref{fig:singularity}. We observe that whilst the solution always blows
up in finite time, the blow up time is later for higher $r$, and is
roughly proportional to $r$, as predicted by the asymptotic approximation.
This is also compatible with the intuition that the $r\to \infty$ limit
should preserve the $\sob{1}{\infty}$ norm.
\begin{figure}
    \centerline{\includegraphics[width=15cm]{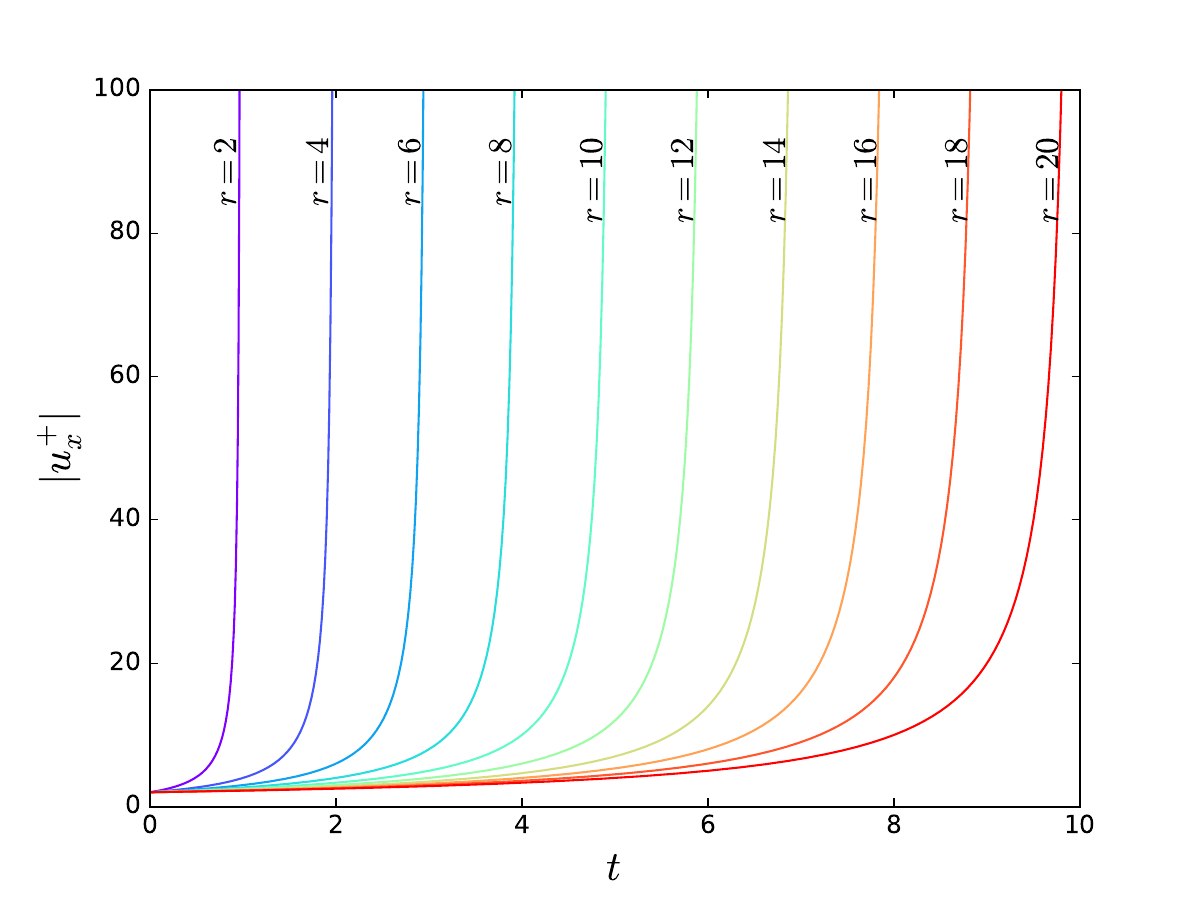}}
    \caption{\label{fig:singularity}Plot of $|u_x|$ in the region $[Q_1,1]$,
      \revised{where it is independent of $x$ and denoted $|u_x^+|$},
      against time for various $r$, with the same initial condition \revised{as Figure
        \ref{fig:r2-6p1},}
      $(Q_1,\widehat{u})=(0.1,0.1)$. The time to singularity increases with $r$.}
\end{figure}

For the single point case, the $r\to\infty$ limit suggested in
Corollary \ref{thm:rinfty}
results in the equations
\begin{equation}
  \revised{z_1 = \max\left(\frac{\widehat{u}_1}{Q_1},\frac{\widehat{u}_1}{1-Q_1}\right), \quad
    \dot{z}_1=0.}
\end{equation}
We also deduce that \revised{$\dot{z}_1=0$} from taking the limit in the
\revised{$P_1$}-equation. For initial conditions $Q_1(0)>0.5$, we obtain
\begin{equation}
  Q_1(t) = 1 - (1-Q_1(0))e^{-\revised{z_1}t},
\end{equation}
\revised{for $z_1$ constant and positive. This means that $z_1$}
does not \revised{reach 1} in finite time, and hence the solution does
not blow up in finite time. The infinite limit solutions are compared
with the finite $r$ solutions in Figure \eqref{fig:infty}.
\begin{figure}
  \centerline{\includegraphics[width=9cm]{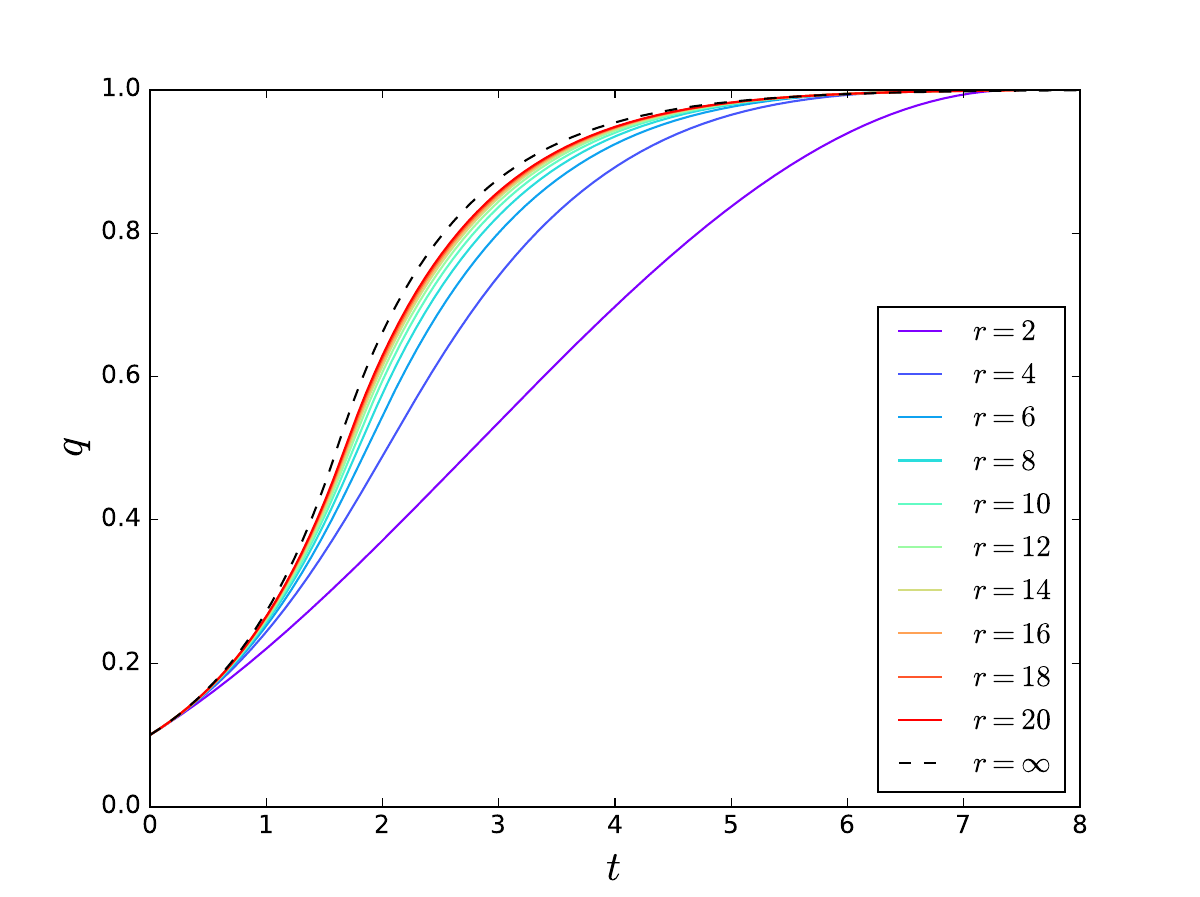}
    \includegraphics[width=9cm]{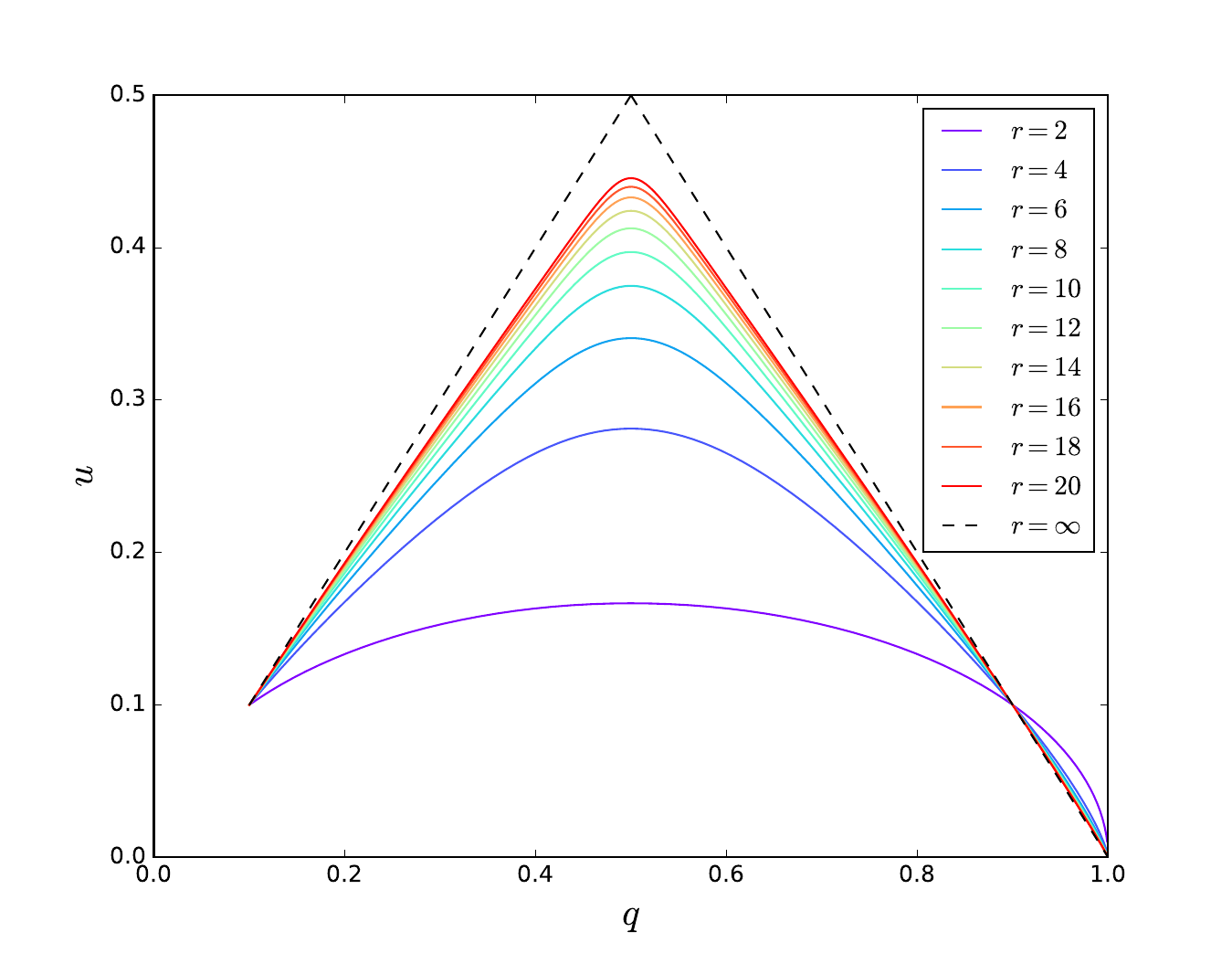}}
  \caption{\label{fig:infty}Comparison of $r\to\infty$ limiting
    solutions with finite $r$ solutions, as per Figure
    \ref{fig:r2-6p1}.}
\end{figure}

\subsection{Two point solutions}
For two or more points, unless a symmetric solution is sought, there is
no closed form for the \revised{$\widehat{u}_1-Q_1$} relation and equation \eqref{eq:hatu}
must be solved iteratively using Newton's method. The result can then
be used in an explicit time-integrator such as the standard 4th-order
Runge-Kutta scheme that we used for these examples. In this section,
we consider three cases, as follows.
\begin{enumerate}
\item Symmetric collision: $(\revised{Q_1},\revised{\widehat{u}_1})=(0.1,0.1)$, $(\revised{Q_2},\revised{\widehat{u}_2})=(0.9,-0.1)$,
\item Chasing collision: $(\revised{Q_1},\revised{\widehat{u}_1})=(0.1,0.2)$, $(\revised{Q_2},\revised{\widehat{u}_2})=(0.2,0.1)$, and
\item Asymmetric collision: $(\revised{Q_1},\revised{\widehat{u}_1})=(0.1,0.2)$, $(\revised{Q_2},\revised{\widehat{u}_2})=(0.2,-0.125)$.
\end{enumerate}
The numerical solutions are shown in Figure \ref{fig:2pt}. We note
that finite time singularities occur in all three cases, even the
chasing collision (in contrast to the same case for the Camassa-Holm
equation, where the solitons transfer momentum at a distance and the
solution does not blow up).
\begin{figure}
\centerline{\includegraphics[width=9cm]{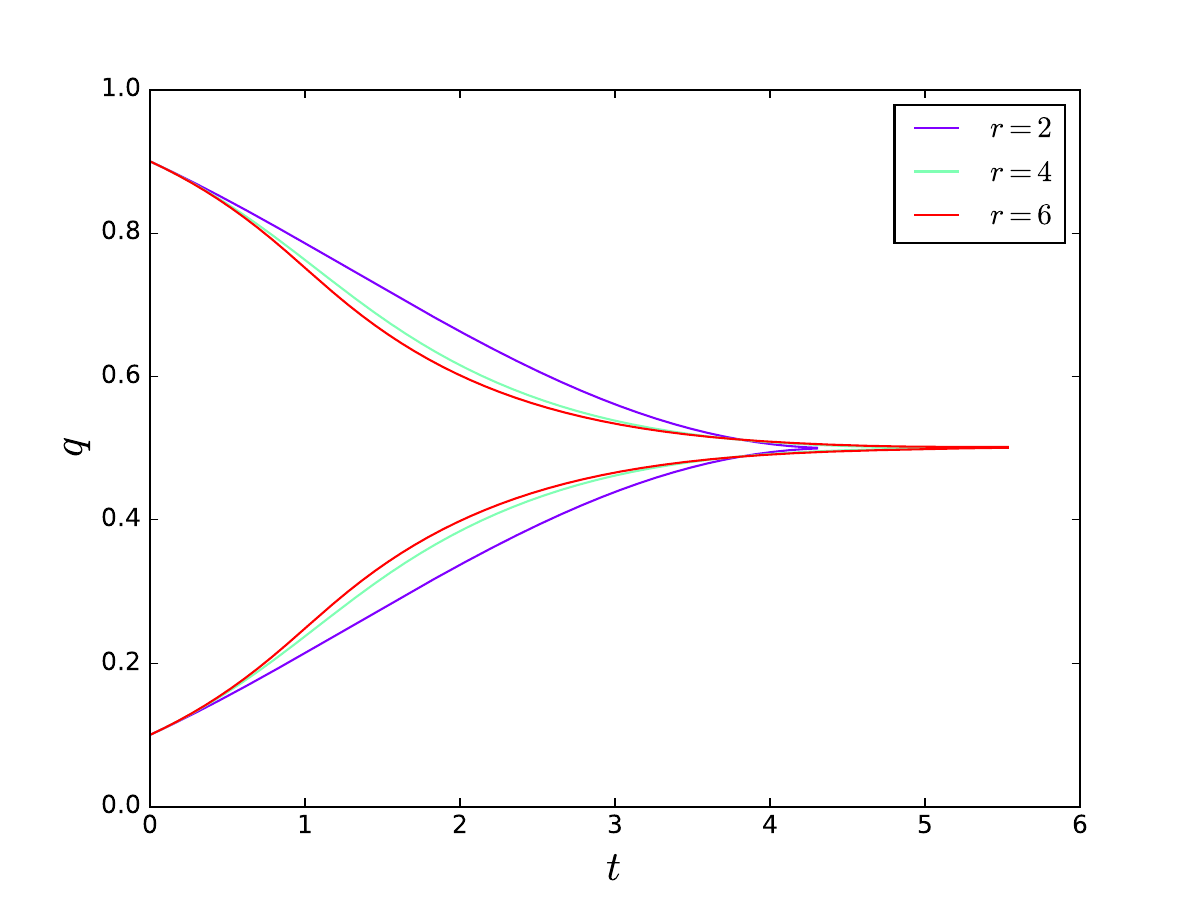}
  \includegraphics[width=9cm]{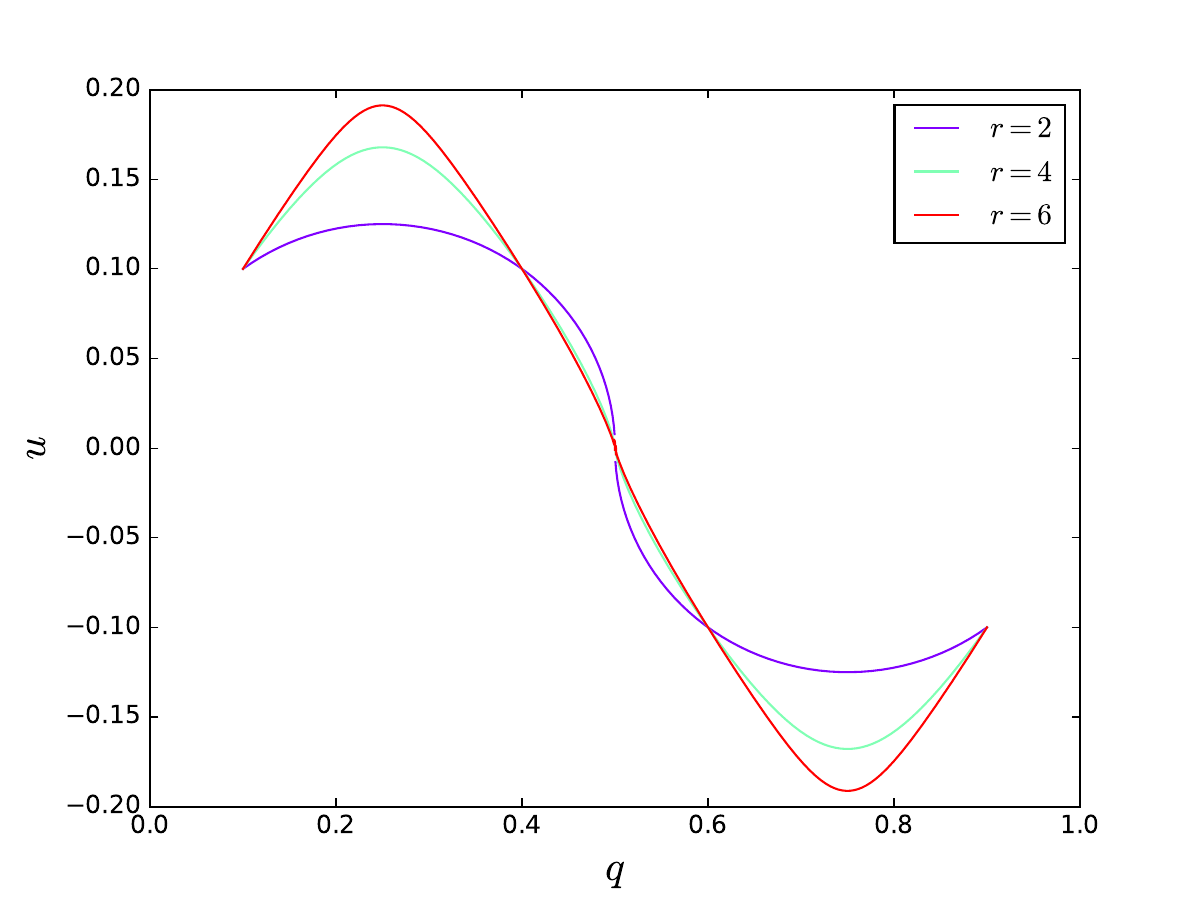}}
\centerline{\includegraphics[width=9cm]{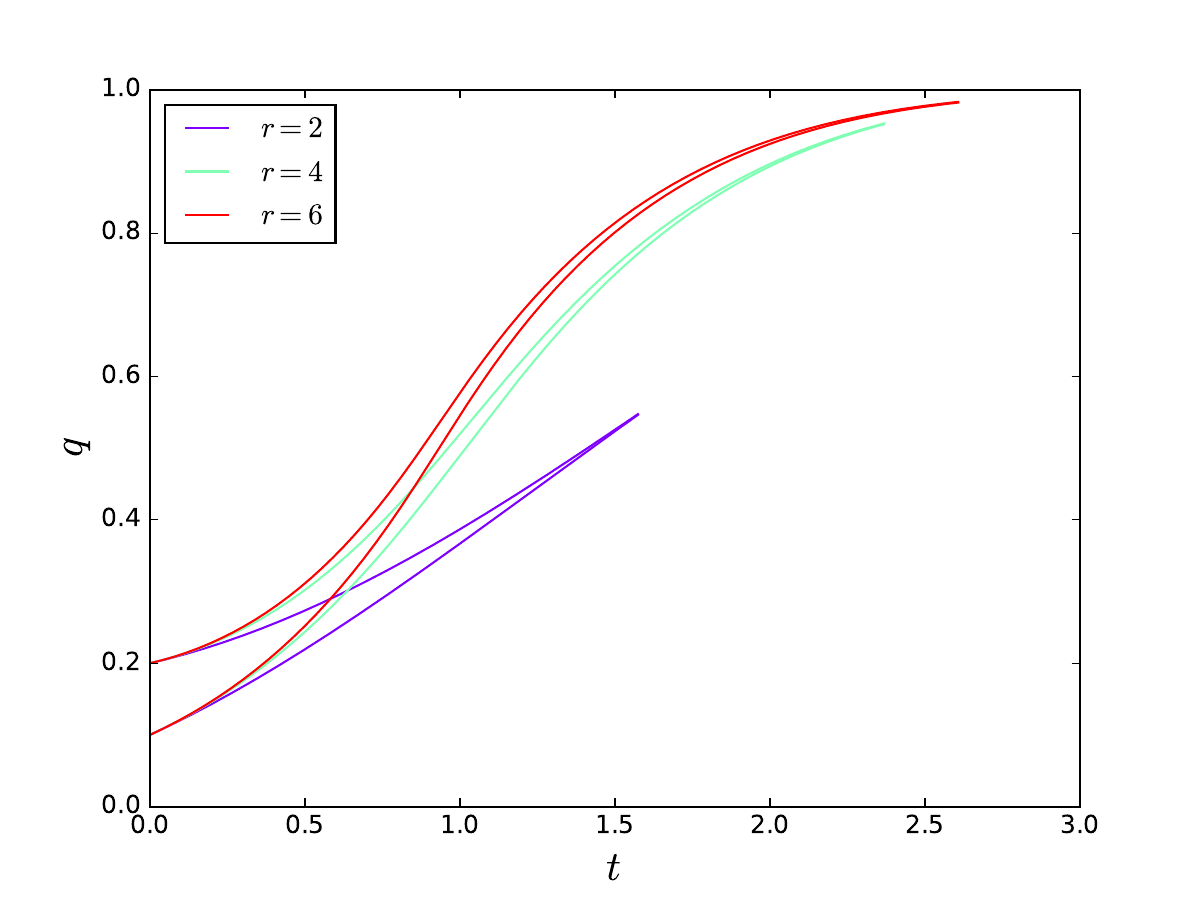}
  \includegraphics[width=9cm]{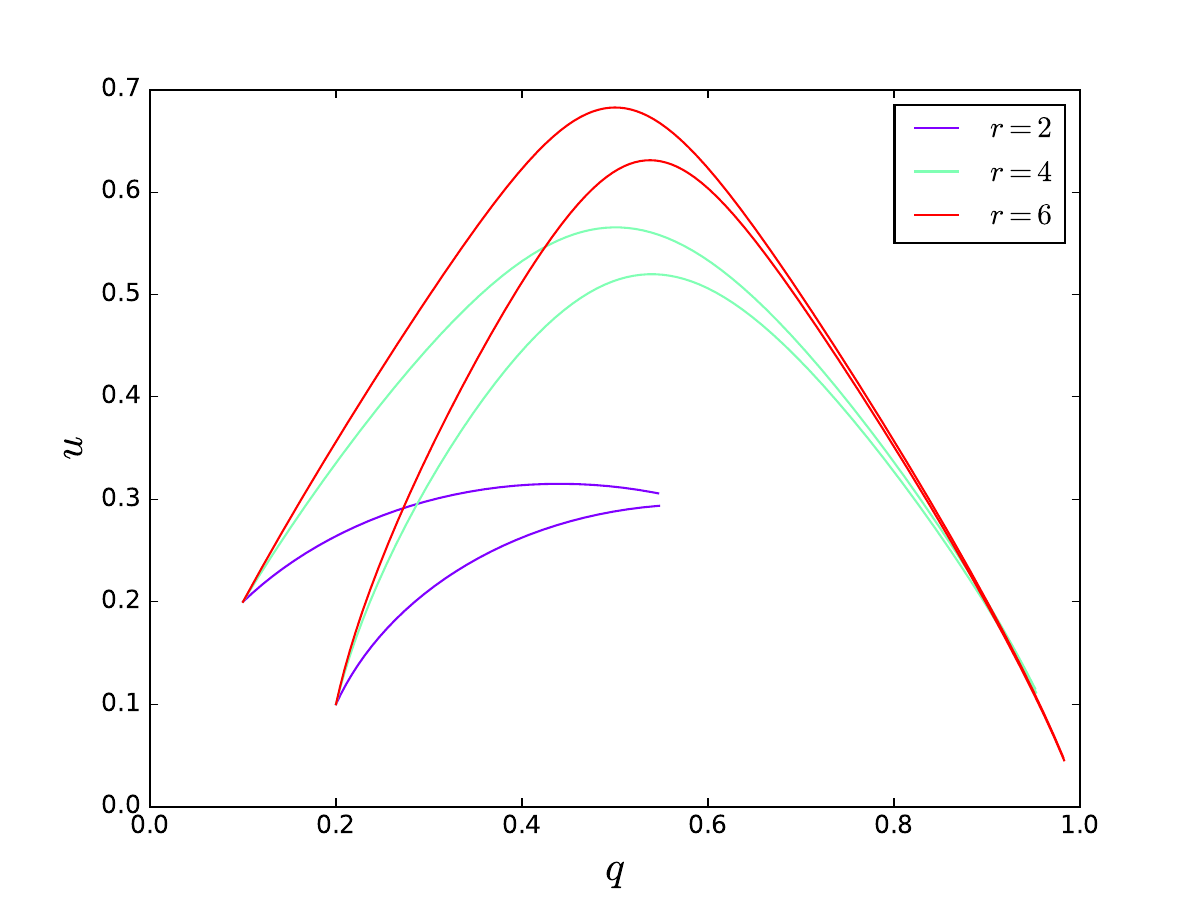}}
\centerline{\includegraphics[width=9cm]{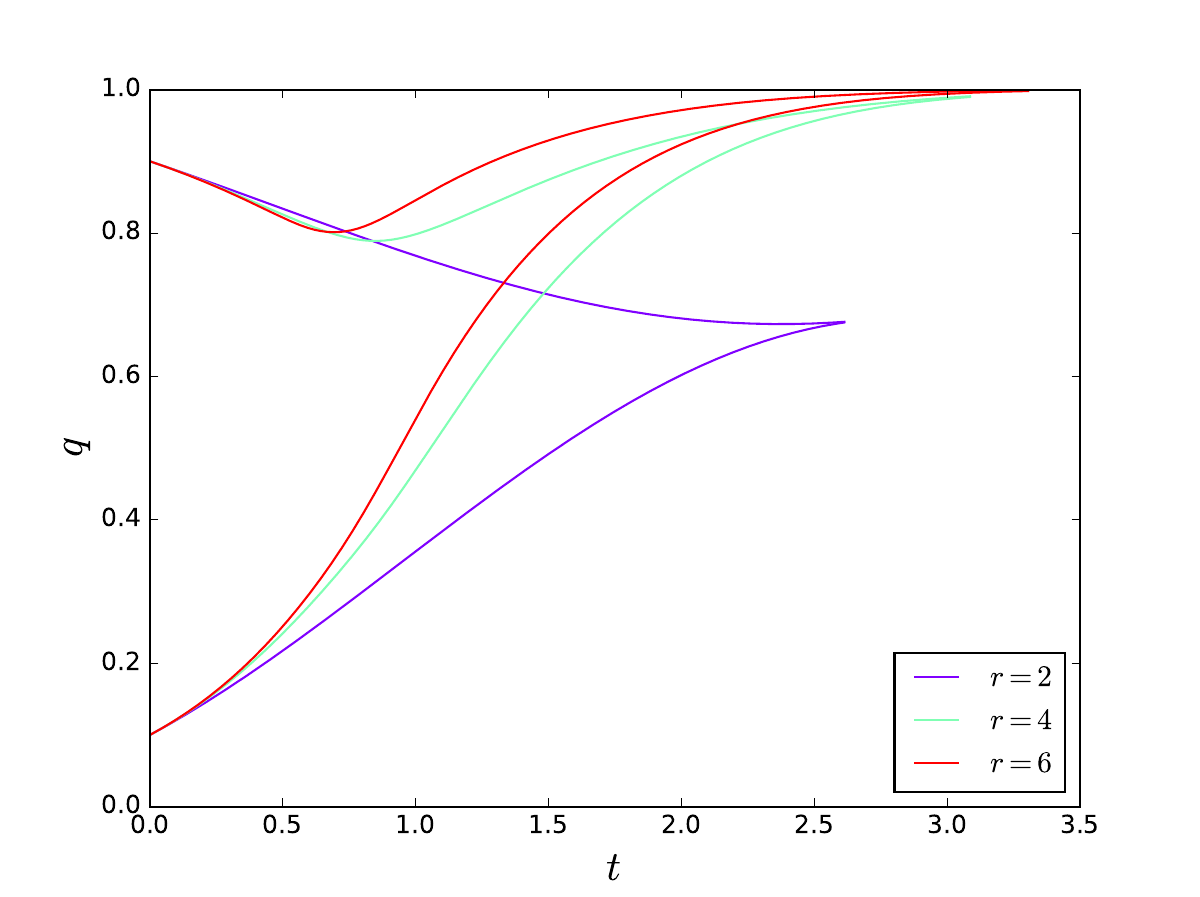}
    \includegraphics[width=9cm]{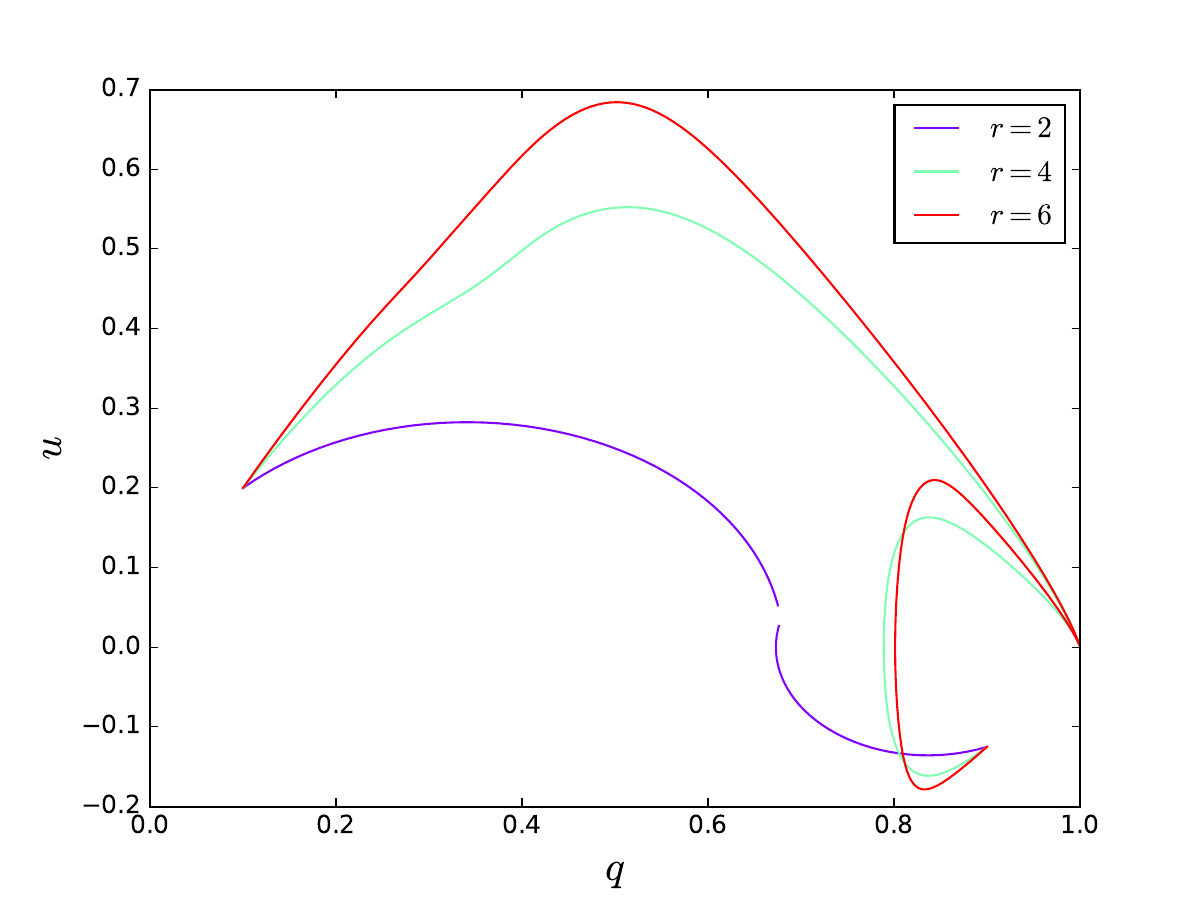}}
\caption{\label{fig:2pt}Two point solutions.  For each $r$ value, the
  trajectories for both points \revised{$(u,q) = (\hat{u}_i,Q_i)$} are
  shown in the same colour.  Top: Symmetric collision.  \revised{The
    points have equal velocities with opposite sign, and are on course
    to a head-on collision. On the left are $(q,t)$ plots, with
    $q=Q_i$ for $i=1,2$, showing collisions at progressively later
    times for larger $r$. On the right are phase plane plots with
    $u=\hat{u}_i(t)$ plotted versus $q=Q_i(t)$ We see that the
    velocity starts decaying earlier for larger $r$, but not decaying
    to zero so a collision always occurs in finite time.}  Middle:
  Chasing collision. \revised{Both points are moving to the right, but
    the one behind is moving faster, and eventually catches up in a
    collision.}  Bottom: Asymmetric collision. \revised{The points
    have velocities of opposite sign, with the point on the left
    moving faster, so the solution is not symmetric. In this solution,
    a new feature emerges for $r>2$, with the initially left-moving
    slower point on the right eventually moving right before eventual
    collision. This contrasts with the $r=2$ case, where the left-moving
    point does not undergo a reversal.} }
\end{figure}
Finally, we investigate the approximation of smooth solutions by
piecewise-linear solutions and their subsequent evolution in time.
We take 101 equispaced points on the interval $[0,1]$ and interpolate
the function $u(x)=\sin(2\pi x)$ to produce an initial condition
for $\{Q_i\}_{i=0}^{100}$, $\{P_i\}_{i=0}^{100}$. We observe a strong
jump in the derivative emerging for large $r$. The results are shown
in Figure \ref{fig:Smooth}.
\begin{figure}
  \centerline{\includegraphics[width=9cm]{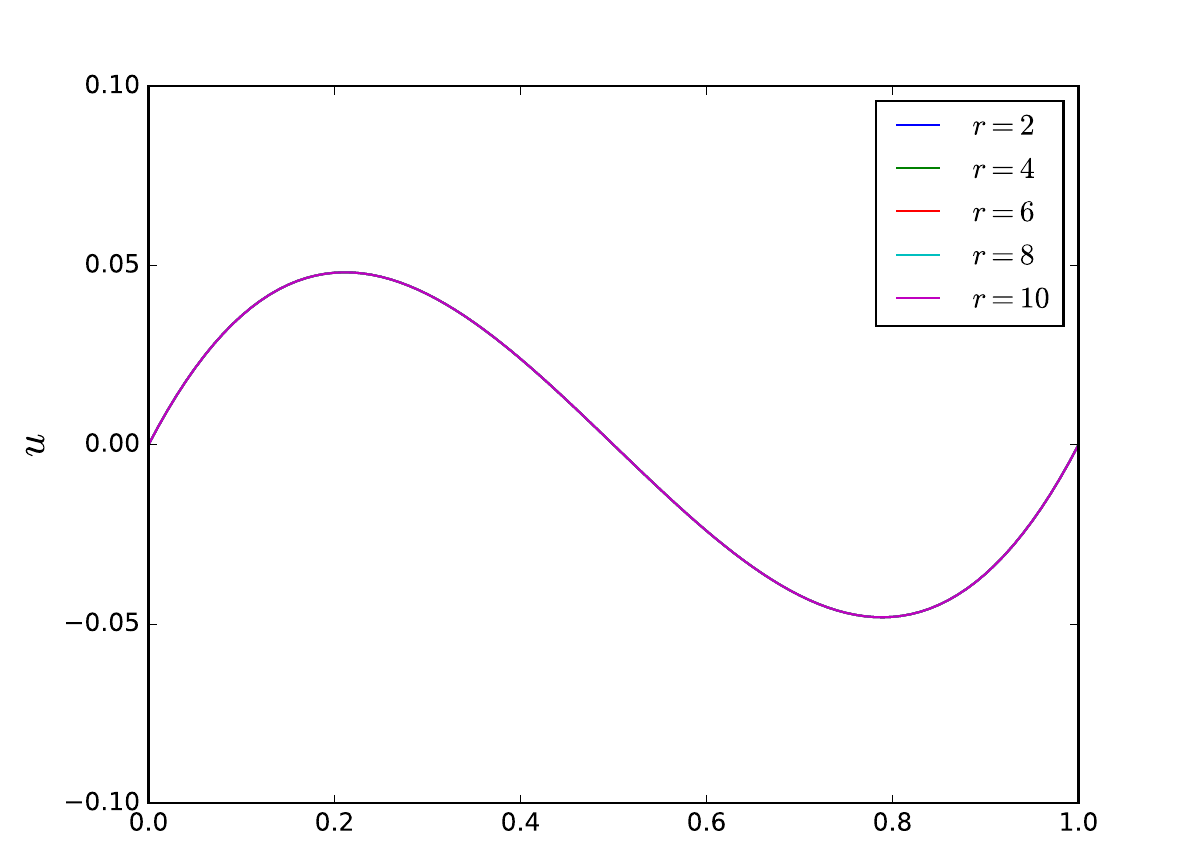}
    \includegraphics[width=9cm]{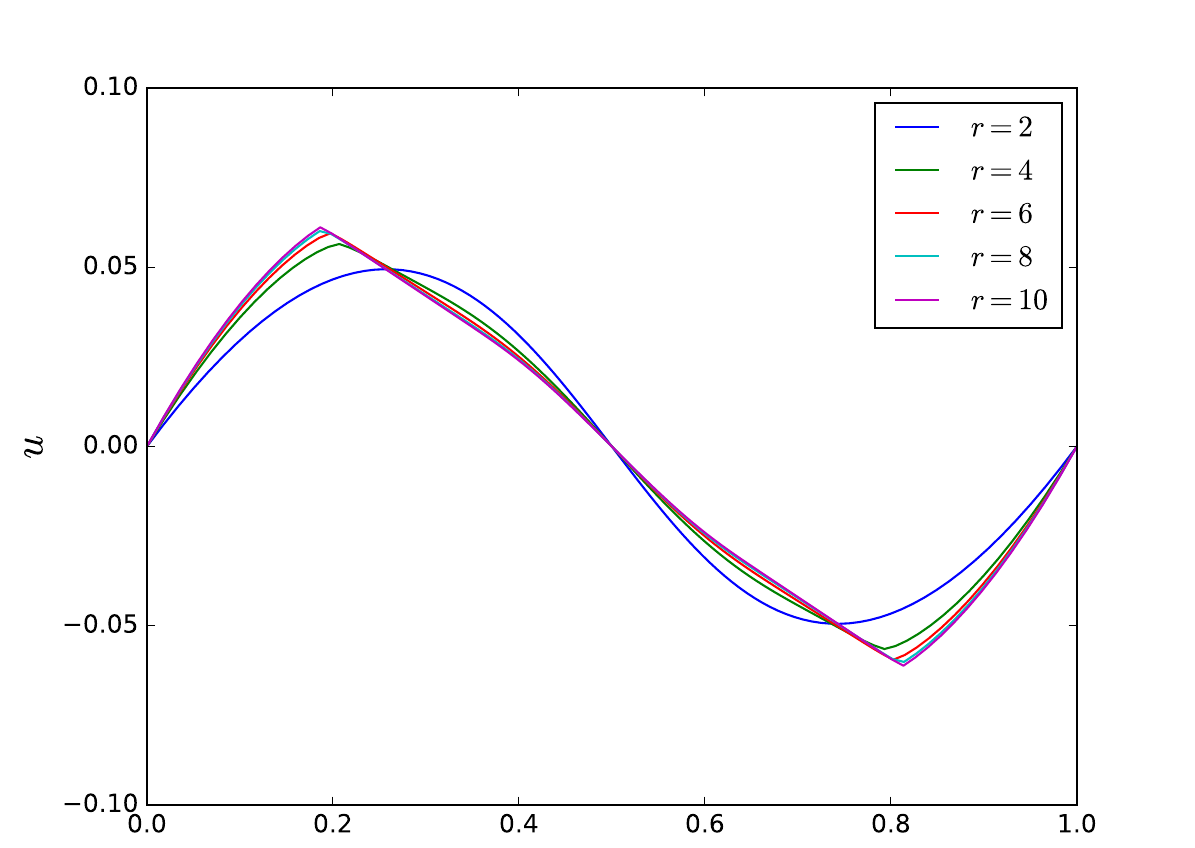}}
    \centerline{\small$x\hspace{90mm}x$}
  \centerline{\includegraphics[width=9cm]{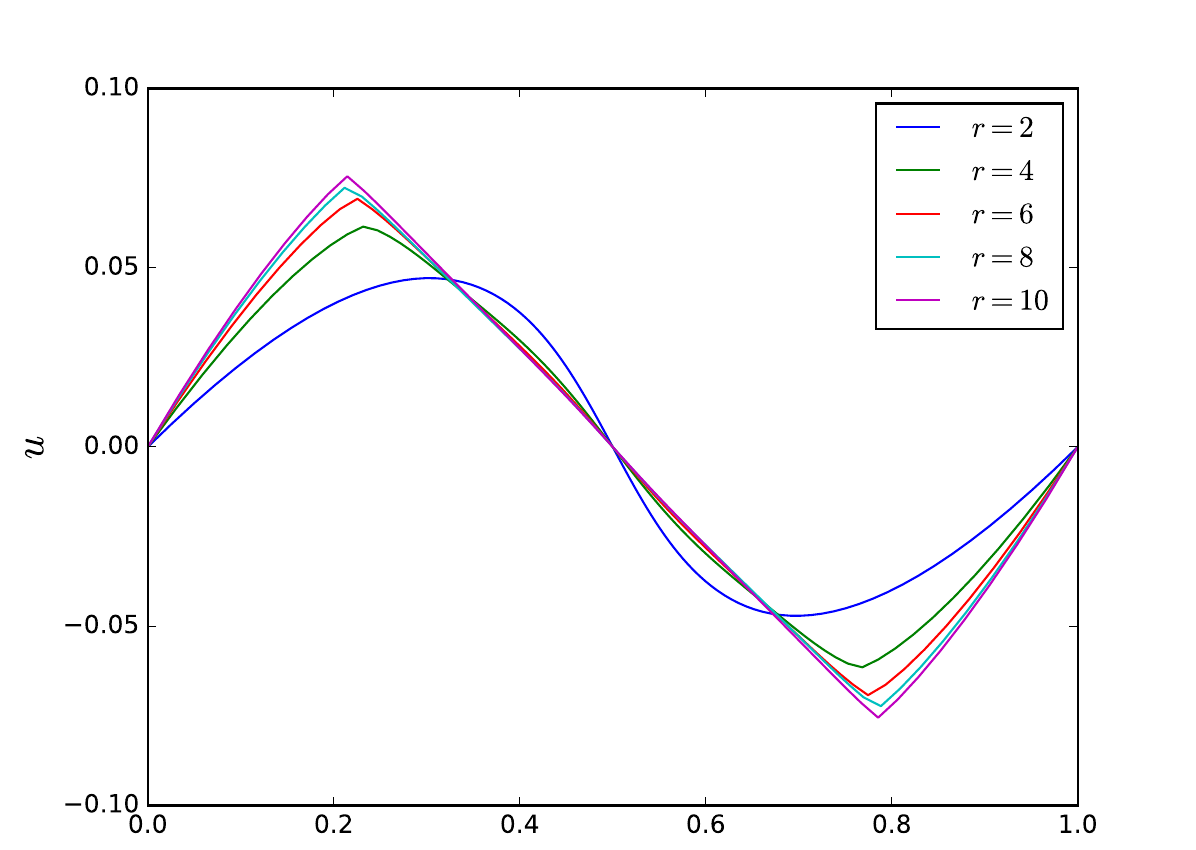}
    \includegraphics[width=9cm]{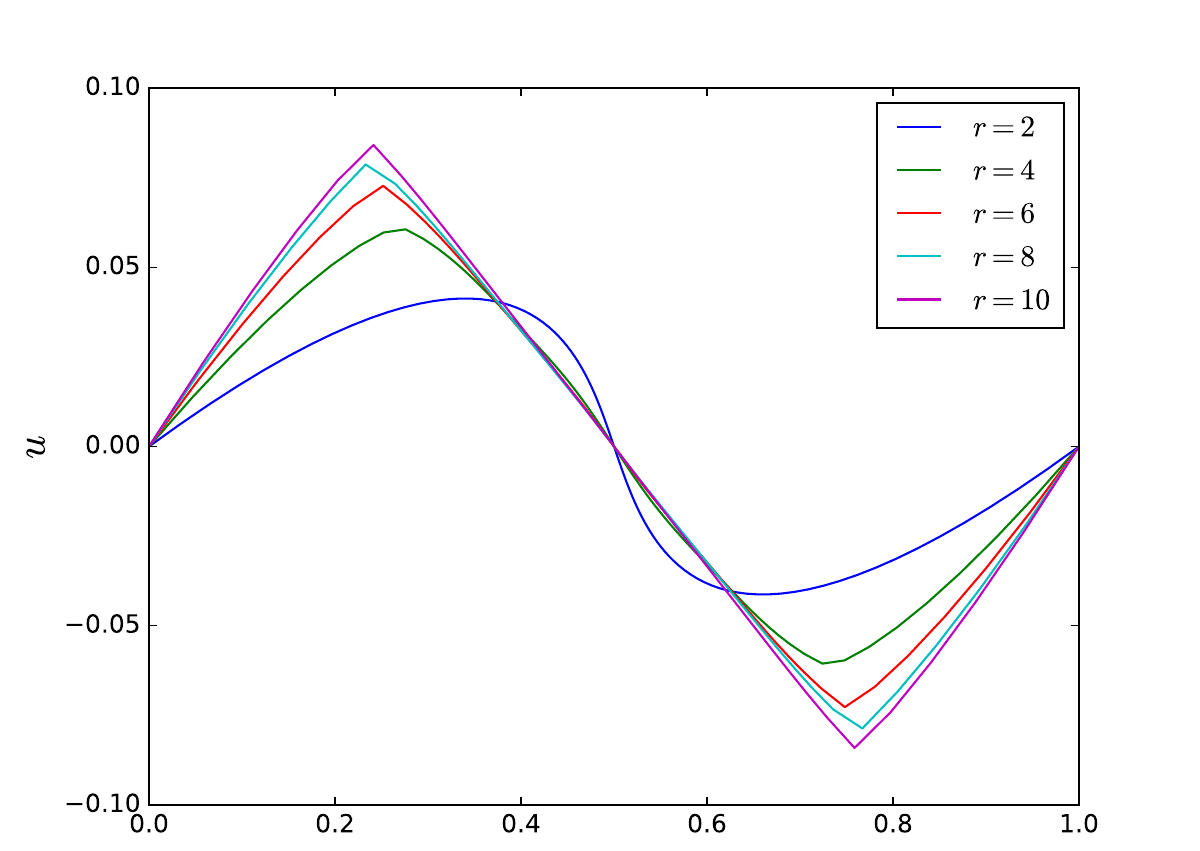}}
    \centerline{\small$x\hspace{90mm}x$}
  \caption{\label{fig:Smooth}Evolution of piecewise linear solutions \revised{$u(x)$}
    initialised from interpolation of smooth functions, shown at times
    $t=0$, $1.33$, $2,67$, $4$. A strong jump in the derivative is emerging
    for large $r$.}
\end{figure}

\section{Summary and outlook}
\label{sec:outlook}

In this paper we introduced an extension of the Hunter-Saxton
equation, which we call the $r$-Hunter-Saxton equation. These
equations are the Euler-Poincar\'e equations with reduced Lagrangian
given by the $\sob{1}{r}$ norm, and are associated with the evolution of
geodesics in the diffeomorphism group with metric defined by the
$\sob{1}{r}$ norm. This replaces the linear Laplacian relating momentum and
velocity by a nonlinear $r$-Laplacian. We introduced an optimal control
variational principle for parameterising finite dimensional subspaces
of the solutions of PDEs, and derived a Hamiltonian system for a
finite dimensional set of points $Q_1,\ldots,Q_n$ plus their conjugate
momenta. This corresponds to piecewise linear functions $u$ with jumps
in the derivative at each of these points. Although these piecewise
functions are not smooth enough to satisfy the $r$-Hunter-Saxton
equation, we showed that remarkably they are still weak solutions of
an integrated form of the $r$-Hunter-Saxton equation.

There are plenty of interesting open questions about the
$r$-Hunter-Saxton equations and the piecewise linear solutions in the
limit as $r\to \infty$.  We have presented some evidence that
solutions of Equations (\ref{eq:Qdot}-\ref{eq:hatu}) converge to
solutions of Equations (\ref{eq:Qz}-\ref{eq:uz}) (before blow-up). We
have also presented evidence that solutions of (\ref{eq:Qz}-\ref{eq:uz})
exist for all times, whilst solutions of  (\ref{eq:Qdot}-\ref{eq:hatu})
blow up in finite time.

It is very interesting to link the solutions to
(\ref{eq:Qz}-\ref{eq:uz}) back to the $r\to \infty$ limit of the
optimal control problem in Definition \ref{def:optimal}. This is
dangerous, since it involves exchanging the limits $r\to \infty$ and
the limits defining variational derivatives. However, it is
tantalising that the solutions to (\ref{eq:Qz}-\ref{eq:uz}) preserve
their $\sob{1}{\infty}$ norm.

Another interesting question is whether solutions of
(\ref{eq:Qdot}-\ref{eq:hatu}) can be used to approximate smooth
solutions of the $r$-Hunter-Saxton equation. The control provided
over the $\sob{1}{p}$ norm of the numerical solution should provide a
useful tool for this.

A final direction of enquiry is, what is the correct $r\to \infty$
limit of the $r$-Hunter-Saxton equation? Can the solutions of the
limiting equation be approximated by (\ref{eq:Qz}-\ref{eq:uz})?  Does
the limiting solution preserve the $\sob{1}{\infty}$ norm? If so this would
provide a way to establish existence for arbitrary time intervals. Do
the limiting solutions generate geodesics for diffeomorphisms in 1D
under the $\sob{1}{\infty}$-norm? \revised{A useful starting point in this
  direction might be the recent work of \cite{bauer2019can}, who
  extended the work of \cite{lenells2007hunter} to $\sob{1}{r}$ metrics
  through mappings to the $\leb{r}$-sphere, where the $r\to \infty$ limit
  may be treated more easily.}

All of these intriguing questions will be the subject of future work.

\revised{\section*{Acknowledgements} The authors are grateful to the two anonymous
referees whose careful reading and useful comments significantly improved
the clarity of this paper. In particular, one of the referees pointed out the recovery
of Burgers' equation in the $r\to 1$ limit.}

\bibliographystyle{alpha}
\bibliography{phs,tristansbib,tristanswritings}

\end{document}